\documentclass{article}
\usepackage[utf8]{inputenc}
\usepackage{fullpage}
\usepackage{amsmath}
\usepackage{amssymb}
\usepackage{amsfonts}
\usepackage{tikz-cd}
\usepackage{amsthm}
\usepackage{comment}
\usepackage{enumitem}
\usepackage{bm}
\usepackage{appendix}

\usepackage{hyperref}
\hypersetup{
    colorlinks=true,
    linkcolor=blue,
    filecolor=magenta,      
    urlcolor=cyan,
}

\newtheorem{theorem}{Theorem}[section]
\newtheorem{lemma}[theorem]{Lemma}
\newtheorem{proposition}[theorem]{Proposition}
\newtheorem{corollary}[theorem]{Corollary}

\newtheorem*{ref_theorem}{Theorem}

\theoremstyle{definition}
\newtheorem{definition}[theorem]{Definition}
\newtheorem{remark}[theorem]{Remark}
\newtheorem{example}[theorem]{Example}
\newtheorem{problem}{Problem}

\theoremstyle{plain}

\newcommand{\C}{\mathbb{C}}
\newcommand{\R}{\mathbb{R}}
\newcommand{\N}{\mathbb{N}}
\newcommand{\K}{\mathbb{K}}
\newcommand{\CP}{\mathbb{CP}}
\newcommand{\UHP}{\mathcal{H}_+}
\newcommand{\LHP}{\mathcal{H}_-}
\newcommand{\D}{\mathbb{D}}

\DeclareMathOperator{\Hmg}{Hmg}
\DeclareMathOperator{\Symb}{Symb}
\DeclareMathOperator{\Hom}{Hom}
\DeclareMathOperator{\ev}{ev}
\DeclareMathOperator{\maxroot}{maxroot}
\DeclareMathOperator{\minroot}{minroot}
\DeclareMathOperator{\SL}{SL}

\title{A Representation Theoretic Explanation of the Borcea-Br{\"a}nd{\'e}n Characterization}
\author{Jonathan Leake}

\begin{document}

\maketitle

\begin{abstract}
    In 2009, Borcea and Br{\"a}nd{\'e}n characterized all linear operators on multivariate polynomials which preserve the property of being non-vanishing (stable) on products of prescribed open circular regions. We give a representation theoretic interpretation of their findings, which generalizes and simplifies their result and leads to a conceptual unification of many related results in polynomial stability theory. At the heart of this unification is a generalized Grace's theorem which addresses polynomials whose roots are all contained in some real interval or ray. This generalization allows us to extend the Borcea-Br{\"a}nd{\'e}n result to characterize a certain subclass of the linear operators which preserve such polynomials.
\end{abstract}

\section{Introduction}

In 1914, P{\'o}lya and Schur \cite{polya1914zwei} characterized the set of diagonal linear operators on polynomials which preserve real-rootedness. Since this seminal paper, much work has been done in extending this characterization to other classes of linear operators. This program in essence came to a close in 2009 with a paper of Borcea and Br{\"a}nd{\'e}n \cite{bb1}, which gave a complete characterization of linear operators on polynomials which preserve real-rootedness.

Their real-rootedness preservation characterization is derived from a more general result pertaining to stable polynomials. Given $\Omega \subset \C^m$, we say that a polynomial $f \in \C[x_1,\ldots,x_m]$ is \emph{$\Omega$-stable} if $f$ does not vanish in $\Omega$. Further, $f$ is \emph{real stable} if it has real coefficients and is $\UHP^m$-stable, where $\UHP \subset \C$ is the open upper half-plane. (We also denote the open lower half-plane by $\LHP \subset \C$.) We additionally use the terms \emph{weakly} $\Omega$-stable and \emph{weakly} real stable if we allow $f \equiv 0$. Finally, we write $f \in \C^\lambda[x_1,\ldots,x_m]$ for $\lambda \equiv (\lambda_1,\ldots,\lambda_m)$ if $f$ is of degree at most $\lambda_k$ in $x_k$. We are then led to the following problems for $\K \in \{\C,\R\}$, generalized from the P{\'o}lya-Schur characterization:

\begin{problem}
    Characterize linear operators $T: \K^\lambda[x_1,\ldots,x_m] \to \K[x_1,\ldots,x_m]$ preserving weak $\Omega$-stability.
\end{problem}

\begin{problem}
    Characterize linear operators $T: \K[x_1,\ldots,x_m] \to \K[x_1,\ldots,x_m]$ preserving weak $\Omega$-stability.
\end{problem}

In \cite{bb1}, Borcea and Br{\"a}nd{\'e}n were able to solve these problems in many cases. In particular, they solved both problems for $\K = \R$ and $\Omega = \UHP^m$, where $m=1$ corresponds to the case of preservation of real-rooted polynomials. For $\K = \C$, they were able to solve Problem 1 for $\Omega$ that is \emph{any} product of open circular regions in $\C$.

In this paper we will only be concerned with Problem 1, for which we now state the solution from \cite{bb1}. Given a linear operator $T: \K^\lambda[x_1,\ldots,x_m] \to \K[x_1,\ldots,x_m]$, a polynomial $\Symb_{BB}(T)$ called the \emph{(Borcea-Br{\"a}nd{\'e}n) symbol} is associated to $T$. Specifically, the symbol is a polynomial in $\K^{\lambda \sqcup \lambda}[x_1,\ldots,x_m, z_1,\ldots,z_m]$ (i.e., of $2m$ variables), where $\lambda \sqcup \lambda := (\lambda_1,\ldots,\lambda_m,\lambda_1,\ldots,\lambda_m)$. The crucial feature of the symbol is that it shares certain stability properties with its associated linear operator, which yields the characterizations stated in the following results. (We will express these results in more detail in \S\ref{Cops_sect} and \S\ref{Rops_sect}.)

\begin{theorem}[Borcea-Br{\"a}nd{\'e}n]\label{BB_Cops_thm}
    Fix a linear operator $T: \C^\lambda[x_1,\ldots,x_m] \to \C[x_1,\ldots,x_m]$ which has image of dimension greater than one. Then, $T$ maps $\UHP^m$-stable polynomials to weakly $\UHP^m$-stable polynomials if and only if $\Symb_{BB}(T)$ is $\UHP^{2m}$-stable.
\end{theorem}

\begin{theorem}[Borcea-Br{\"a}nd{\'e}n]\label{BB_Rops_thm}
    Fix a linear operator $T: \R^\lambda[x_1,\ldots,x_m] \to \R[x_1,\ldots,x_m]$ which has image of dimension greater than two. Then, $T$ maps real stable polynomials to weakly real stable polynomials if and only if either $\Symb_{BB}(T)$ or $\Symb_{BB}(T^-)$ is real stable, where $T^-(p) := T(p(-x_1,\ldots,-x_m))$.
\end{theorem}

To deal with other products of circular regions, one then conjugates $T$ by certain M{\"o}bius transformations and applies Theorem \ref{BB_Cops_thm} to the conjugated operator. Unfortunately though, this is a tedious process which has to be done each time a new stability region is to be considered. Additionally, the image dimension restrictions give rise to degeneracy cases which have to be dealt with separately. Both of these issues obscure the connection between an operator and its symbol.

In this paper, we present a new conceptual approach to the Borcea-Br{\"a}nd{\'e}n characterization via the representation theory of $\SL_2(\C)$. In particular, we derive a new symbol (denoted $\Symb$) in a natural way, and our definition eliminates the issues discussed above. This is seen in the following results, which are our simplified and generalized versions of the Borcea-Br{\"a}nd{\'e}n characterizations. Note that for the sake of simplicity, we have omitted a few details here regarding non-convex circular regions. Specifically, circular regions of $\C$ should be thought of as lying in the Riemann sphere, so that complements of discs contain the point at $\infty$.

\newtheorem*{Cops_graces}{Theorem \ref{Cops_graces}}
\begin{Cops_graces}
    Fix a linear operator $T: \C^\lambda[x_1,\ldots,x_m] \to \C^\alpha[x_1,\ldots,x_l]$, a product of all open or all closed circular regions $\Omega_0 = C_1 \times \cdots \times C_m$, and a product of sets $\Omega_1 := S_1 \times \cdots \times S_m$. Further, denote $\widetilde{\Omega}_0 := (\C \setminus C_1) \times \cdots \times (\C \setminus C_m)$. Up to certain degree and convexity (of $C_k$) restrictions, we have that $T$ maps $\Omega_0$-stable polynomials to nonzero $\Omega_1$-stable polynomials if and only if $\Symb(T)$ is $(\widetilde{\Omega}_0 \times \Omega_1)$-stable.
\end{Cops_graces}

\newtheorem*{Rops_graces}{Theorem \ref{Rops_graces}}
\begin{Rops_graces}
    Fix a linear operator $T: \R^\lambda[x_1,\ldots,x_m] \to \R^\alpha[x_1,\ldots,x_l]$. Up to certain degree restrictions, $T$ maps real stable polynomials to nonzero real stable polynomials if and only if $\Symb(T)$ is either $(\overline{\LHP}^m \times \UHP^l)$-stable or $(\overline{\LHP}^m \times \LHP^l)$-stable.
\end{Rops_graces}

We summarize the specific improvements that this and our other related results give over the Borcea-Br{\"a}nd{\'e}n characterization as follows.

\begin{enumerate}
    \item \emph{Different stability regions can be considered using the same symbol.} The symbol we define in this paper is universal: for example, it gives stability-preservation information for \emph{any} product of open circular regions. The Borcea-Br{\"a}nd{\'e}n symbol, on the other hand, required the application of M{\"o}bius transformations. In addition, our symbol also allows for the output stability region to be chosen independently of the input stability region. While this does not literally improve the result, it does allow for quicker computations. In particular, see Examples \ref{add_conv_ex1} and \ref{mult_conv_ex1} where classical polynomial convolution results are easily derived from our framework.
    
    \item \emph{Our characterization does not require any degeneracy condition.} Our results characterize operators which preserve (strong) stability rather than weak stability. As seen above, this slightly stronger notion of stability enables us to eliminate any image dimension degeneracy condition, as required in the Borcea-Br{\"a}nd{\'e}n characterizations (Theorems \ref{BB_Cops_thm} and \ref{BB_Rops_thm}). This demonstrates a cleaner link between an operator and its symbol.
    
    \item \emph{Closed circular regions and projectively convex regions can be considered.} The symbol we define in this paper handles products of open circular regions, as well as products of closed circular regions. (In \cite{melamud2015ops}, Melamud proves a result similar to the Borcea-Br{\"a}nd{\'e}n characterization for closed circular regions.) Further, we are also able to consider more general \emph{projectively convex} regions (circular regions with portions of their boundary; also called \emph{generalized circular regions}, see \cite{zervos1960aspects} and \cite{zaheer1976polar}) in Proposition \ref{Cops_prop}. This allows us to determine stability-preservation information about real intervals and half-lines. It also turns out, somewhat surprisingly, that our symbol can handle products of any \emph{sets} as possible output space stability regions (as seen in Theorem \ref{Cops_graces} above).
\end{enumerate}

In the process of generalizing the Borcea-Br{\"a}nd{\'e}n characterization we develop a general algebraic framework which also encompasses many of the classical polynomial tools. This framework aims to motivate classical results and provide intuition for the connection between a stability preserving operator and its symbol.

\subsubsection*{The Main Idea}

A major purpose of this paper is to explain a certain conceptual thread in the history of polynomial stability theory: that it is often possible to determine general stability information from restricted sets of polynomials. For example, the P{\'o}lya-Schur and Borcea-Br{\"a}nd{\'e}n characterizations derive from a single polynomial (i.e., the symbol) stability properties of a whole collection of polynomials in the output of a given linear operator. Additionally, the Grace-Walsh-Szeg{\H{o}} coincidence theorem says that stability information of any polynomial can be determined from its polarization, which is of degree at most one in each variable.

% We now begin to approach these phenomena from elementary algebraic and representation theoretic concepts.
As it turns out, these sorts of phenomena can be explained using relatively basic algebraic and representation theoretic concepts. We view $\C^n[x]$ as a representation of $\SL_2(\C)$ via the standard action, given as follows. For $\phi \in \SL_2(\C)$ and $f \in \C^n[x]$, we define:
\[
    (\phi \cdot f)(x) := f(\phi^{-1}x)
\]
Here $\phi^{-1}$ acts on $x \in \C$ as a M{\"o}bius transformation, or equivalently $\phi$ acts on the roots of $f$. (Similarly, $\C^\lambda[x_1,\ldots,x_m]$ can be viewed as a representation of $(\SL_2(\C))^m$ via this action in each variable.) Under this interpretation, important maps like polarization, projection, the apolarity form, and even the symbol turn out to be invariant under these $\SL_2(\C)$ actions. This leads us to a conceptual thesis: \emph{$\SL_2(\C)$-invariant maps transfer stability information}.

% \begin{thesis}
%     $\SL_2(\C)$-invariant maps transfer stability information.
% \end{thesis}

The goal of this paper is then to explicate and answer the most important question related to this thesis: what does it mean for the symbol map to be $\SL_2(\C)$-invariant and how does it transfer stability information? To answer this, we consider the following standard ideas relating spaces of linear operators to tensor products.

Let $W_1,W_2$ be two finite dimensional representations of a group $G$, and let $\Hom(W_1,W_2)$ denote the space of linear maps from $W_1$ to $W_2$. Then, $\Hom(W_1,W_2) \cong W_1^* \boxtimes W_2$ (the outer tensor product) can viewed as a representation of $G \times G$. If we further have a $G$-invariant bilinear form on $W_1$, then we also have $W_1 \cong W_1^*$. This leads to the following identification:
\[
    \Hom(W_1,W_2) \cong W_1^* \boxtimes W_2 \cong W_1 \boxtimes W_2
\]
If $W_1$ and $W_2$ are spaces of polynomials, each in $m$ variables, then their tensor product $W_1 \boxtimes W_2$ is isomorphic to a larger space of polynomials in $2m$ variables. That is, a linear operator between polynomial spaces $W_1$ and $W_2$ can be associated to some polynomial in double the variables, via the above identification of representations. This is precisely the idea of the symbol of an operator.

Let's see how this works in the univariate case. Consider $\C^n[x]$ as a representation of the group $\SL_2(\C)$, as described above. It is then a standard result that the classical bilinear \emph{apolarity form} is invariant under the action of M{\"o}bius transformations. That is, the apolarity form is an $\SL_2(\C)$-invariant bilinear form on $\C^n[x]$. This form, applied to $f,g \in \C^n[x]$ with coefficients $f_k,g_k$, is defined as follows:
\[
    \langle f,g \rangle^n := \sum_{k=0}^n \binom{n}{k}^{-1} (-1)^k f_k g_{n-k}
\]
With this, we obtain the identification described above: $\Hom(\C^n[x],\C^m[x]) \cong \C^n[x] \boxtimes \C^m[x] \cong \C^{(n,m)}[x,z]$. The final piece of the puzzle is then to find a way to transfer stability information through this identification of representations. The key result to this end is the classical Grace's theorem:

\begin{ref_theorem}[\cite{grace1902zeros}]
    Let $f,g \in \C^n[x]$ be polynomials of degree exactly $n$. Further, let $C$ be some open or closed circular region such that $f$ is $C$-stable and $g$ is $(\C \setminus C)$-stable. Then, $\langle f,g \rangle^n \neq 0$.
\end{ref_theorem}

That is, the apolarity form not only provides the link between a linear operator and its symbol, but also captures stability information. So, whatever stability claims we can make about polynomials in $\C^{(n,m)}[x,z]$ can then be seamlessly transferred to corresponding linear operators in $\Hom(\C^n[x],\C^m[x])$. From this we are able to recover the Borcea-Br{\"a}nd{\'e}n characterization. Additionally, all of the theory here relating stability and the representation theory of $\SL_2(\C)$ can be generalized to multivariate polynomials in a straightforward manner. The details will be discussed in \S\ref{reps_sect}.

In a similar fashion, other important maps also have $\SL_2(\C)$-invariance properties (e.g., polarization and projection, as used in the Grace-Walsh-Szeg{\H{o}} coincidence theorem, explicitly give the isomorphisms of a classical representation theoretic result; see Appendix \ref{gws_app}). A main feature of our conceptual thesis is that it allows for a unification of many seemingly related results in polynomial stability theory. A crucial point to make then is that Grace's theorem is at the heart of this unification. That said, a significant portion of this paper is devoted to discussing it.

\subsubsection*{A Generalized Grace's Theorem and Interval- and Ray-Rootedness}

In \cite{bb2}, Borcea and Br{\"a}nd{\'e}n are able to prove a multivariate Grace's theorem using their operator characterization. In this paper we will prove the multivariate version from scratch, and then use it to derive a new characterization of stability-preserving linear operators. In addition, we generalize it to \emph{projectively convex} regions, which consist of an open circular region with a portion of its boundary (see \S\ref{proj_conv_subsect}). We state our new result as follows. Note that this result can be seen as an extension of the generalized Grace's theorem given in Corollary 4.4 of \cite{zaheer1976polar}.

\newtheorem*{graces_thm}{Theorem \ref{graces_thm}}
\begin{graces_thm}
    Fix $\lambda \in \N_0^m$ and $f,g \in \C^\lambda[x_1,\ldots,x_m]$ such that $f$ and $g$ both have a nonzero term of degree $\lambda$. Also, denote $C := \UHP \cup \overline{\R_+}$ and $\widetilde{C} := \LHP \cup \overline{\R_-}$. If $f$ is $C^m$-stable and $g$ is $\widetilde{C}^m$-stable, then $\langle f,g \rangle^\lambda \neq 0$.
\end{graces_thm}

This result can, for instance, give stability information about positive- and negative-rooted polynomials. Since the apolarity form is invariant under the action of M{\"o}bius transformations, we immediately obtain similar statements regarding the union of any open circular region and a portion of its boundary. Notice also that, unlike the classical Grace's theorem, the stability regions $C$ and $\widetilde{C}$ have non-empty intersection.

In the vein of this extension, we provide a new characterization of a certain class of linear operators which preserve ray- and interval-rootedness. The problem of classifying all such operators is still open in general (see e.g., the end of \cite{bba}). Here, we solve this problem for a restricted class of operators: namely, operators which both preserve weak real-rootedness and also preserve ray- or interval-rootedness. Our main result in this direction is stated as follows, where a polynomial is called $J$-rooted when all of its roots are in $J$:

\newtheorem*{Jops_graces}{Theorem \ref{Jops_graces}}
\begin{Jops_graces}
    Fix a linear operator $T: \R^n[x] \to \R^m[x]$ which has image of dimension greater than two. Further, let $I,J \subseteq \R$ be intervals or rays. Up to certain degree restrictions, $T$ preserves weak real-rootedness and maps $I$-rooted polynomials to nonzero $J$-rooted polynomials if and only if $\Symb(T)$ is either $(\LHP \cup I) \times (\overline{\UHP} \setminus J)$-stable or $(\LHP \cup I) \times (\overline{\LHP} \setminus J)$-stable.
\end{Jops_graces}

In \S\ref{Jops_subsect}, this result is stated in a more restricted manner as the degeneracy condition (image dimension) and degree restrictions end up being more tedious than in the other results. Corollary \ref{Jops_graces_cor} and further explication then give the result as stated here.

As a final note, all of the results given here in the introduction are stated slightly differently in \S\ref{graces_sect}, \S\ref{Cops_sect}, \S\ref{Rops_sect}. In particular, the notation $V(\lambda)$ is used in place of $\C^\lambda[x_1,\ldots,x_m]$, and reference is made to $\CP^1$ (i.e., the Riemann sphere). This notation has to do with consideration of ``roots at infinity'', which allows us to remove degree restrictions and avoid reference to convex circular regions. We discuss this rigorously in \S\ref{homog_subsect}.

\subsubsection*{A Roadmap}

We now describe the content of the remainder of this paper. In \S\ref{prelim_sect}, we discuss the use of homogeneous polynomials via the notation $V(n)$ and $V(\lambda)$, and we describe the relation of these spaces to the notion of roots at infinity.

In \S\ref{reps_sect}, we explicate some very basic representation theory of $\SL_2(\C)$. We then demonstrate how the apolarity form and the symbol arise as natural constructs in this context. Results like the symbol lemma (Lemma \ref{symb_lemma}) and the $\SL_2(\C)$-invariance of the apolarity form are stated here.

In \S\ref{stability_sect}, we discuss some classical and some new polynomial stability theory results, and their multivariate analogues, in the homogeneous context. We also extend Laguerre's theorem to projectively convex regions (generalized circular regions), which later allows us to prove results regarding polynomials which have all their roots in a given interval. 

In \S\ref{graces_sect}, we state and prove our generalized Grace's theorem. We also discuss other stability regions to which the theorem applies, and consider symbols of linear operators given by evaluation at a particular point. We call these polynomials \emph{evaluation symbols}, as they turn out to play a crucial role in the proofs of the operator characterizations.

In \S\ref{Cops_sect}, we finally state and prove our improved characterizations of stability-preserving operators. We also demonstrate how the Borcea-Br{\"a}nd{\'e}n characterizations can be seen (with a bit of work) to be corollaries of our characterizations. We then provide examples of the use of our results. In particular, we show how stability results related to classical polynomial convolutions can be immediately recovered.

In \S\ref{Rops_sect}, we state and prove the analogous characterization of strong real stability-preserving operators. As with complex operators, we show how the Borcea-Br{\"a}nd{\'e}n characterization can be obtained as a corollary. In this section, we also state and prove our characterization of operators which preserve both weak real-rootedness as well as interval- (or ray-) rootedness.

In Appendix \ref{sl2_app}, we explicate more of the representation theory of $SL_2(\C)$ in a polynomial-minded way. Specifically, we prove a few standard tensor product decomposition results which more clearly demonstrate how this theory connects to the notion of apolarity.

In Appendix \ref{gws_app}, we discuss how polarization and the Grace-Walsh-Szeg{\H{o}} coincidence theorem fit in to the framework presented in this paper. We also demonstrate that the classical isomorphism $V(n) \cong \operatorname{Sym}^n(V(1))$ (for representations of $\SL_2(\C)$) can be realized as the polarization map and therefore has a stability-theoretic interpretation. While important to the conceptual thesis stated above, we place this discussion in an appendix as it is not utilized elsewhere.

\section{Preliminaries}\label{prelim_sect}

Here, we discuss basic notation and results related to polynomials and stability. In particular, we discuss in more detail the notation and consequences related to the use of homogeneous polynomials in place of usual univariate and mutlivariate polynomials. Then, we state a number of basic stability results in the language of homogeneous polynomials.

\subsection{Notation}

Let $[n] := \{1,2,\ldots,n\}$ and $(1^m) := (1,\ldots,1) \in \N_0^m$. For $\lambda = (\lambda_1,\ldots,\lambda_m) \in \N_0^m$, we define:
\[
    \C^\lambda[x_1,\ldots,x_m] := \left\{f \in \C[x_1,\ldots,x_m] : \deg_{x_k}(f) \leq \lambda_k\right\}
\]
That is, elements of $\C^\lambda[x_1,\ldots,x_m]$ are of degree at most $\lambda_k$ in the variable $x_k$. In particular, we call polynomials in $\C^{(1^m)}[x_1,\ldots,x_m]$ \emph{multi-affine}. We will also use the shorthand $\C^n[x]$ to refer to univariate polynomials of degree at most $n$.

Now we define similar spaces of polynomials which are homogeneous in pairs of variables. These polynomials should be seen as per-variable homogenizations of polynomials of the spaces defined above. For $\lambda = (\lambda_1,\ldots,\lambda_m) \in \N_0^m$ and $\K = \C$ or $\K = \R$, we define:
\[
    V_\K(\lambda) = V_\K(\lambda_1,\ldots,\lambda_m) := \left\{p \in \K[x_1,y_1\ldots,x_m,y_m] : p \text{ is homogeneous of degree $\lambda_k$ in $x_k$, $y_k$}\right\}
\]
We also use the shorthand $V(\lambda) = V_\C(\lambda)$. As above, we call polynomials in $V(1^m)$ \emph{multi-affine}. The notation used here is generalized from what is typically used to denote the irreducible representations of $\SL_2(\C)$. As it turns out, spaces of homogeneous polynomials in two variables can be used to define these representations. This will be made more precise in \S\ref{reps_sect}.

We let $\CP^1$ denote the projective space of lines in $\C^2$, and we will also identify this space with the Riemann sphere. Note that $\CP^1$ can also be considered as a compact version of $\C$ with one extra point added at infinity. We will often identify $\CP^1$ with $\C$ (up to this extra point) via stereographic projection. A \emph{circular region} is then an open or closed disc, half-plane, or complement of a disc in $\C$. The set of circular regions is transitive under the action of M\"obius transformations. We also use the name circular regions to refer to the stereographic projections into $\CP^1$. Closed half-planes and complements of discs projected into $\CP^1$ will contain the point at infinity. Throughout, we will let $\partial S$ denote the boundary of (the closure of) the set $S$ in $\CP^1$ and let $S^\circ$ denote the interior of $S$.

There is also a natural ordering structure on $\N_0^m$, along with a few basic operations that will be used throughout. Fix $\lambda = (\lambda_1,\ldots,\lambda_m)$ and $\alpha = (\alpha_1,\ldots,\alpha_m)$ in $\N_0^m$, and fix $\beta = (\beta_1,\ldots,\beta_n)$ in $\N_0^l$. We say $\alpha \leq \lambda$ whenever $\alpha_k \leq \lambda_k$ for all $k \in [m]$. We define $\lambda + \alpha := (\lambda_1+\alpha_1,\ldots,\lambda_m+\alpha_m)$, $|\lambda| := \lambda_1 + \cdots + \lambda_m$, and $\lambda \sqcup \beta := (\lambda_1,\ldots,\lambda_m,\beta_1,\ldots,\beta_l) \in \N_0^{m+l}$.

We also make use of a number of shorthands. Fix $\mu,\lambda \in \N_0^m$ such that $\mu \leq \lambda$. We define $x^\mu \in \C_\lambda[x_1,\ldots,x_m]$ via $x^\mu := x_1^{\mu_1}x_2^{\mu_2} \cdots x_m^{\mu_m}$. Similarly, we define $\partial_x^\mu := \partial_{x_1}^{\mu_1} \partial_{x_2}^{\mu_2} \cdots \partial_{x_m}^{\mu_m}$, where $\partial_x := \frac{\partial}{\partial x}$. When considering $V(\lambda)$, we define $y^\mu$ and $\partial_y^\mu$ in the same way. Further, we define $\lambda! := \lambda_1! \cdots \lambda_m!$ and $\binom{\lambda}{\mu} := \frac{\lambda!}{\mu!(\lambda-\mu)!} = \binom{\lambda_1}{\mu_1} \cdots \binom{\lambda_m}{\mu_m}$. Finally, we denote $(-1)^\mu := (-1)^{\mu_1} \cdots (-1)^{\mu_m} = (-1)^{|\mu|}$.

\subsection{Homogeneous Polynomials}\label{homog_subsect}

The usual degree-$n$ homogenization of a polynomial $f \in \C^n[x]$ is defined on monomials as follows and is extended linearly.
\[
    \begin{split}
        \Hmg_n : \C^n[x] &\rightarrow V(n) \\
        x^k &\mapsto x^ky^{n-k}
    \end{split}
\]
More generally, for $\lambda \in \N_0^m$ the \emph{degree-$\lambda$ homogenization} is defined on monomials as follows and is extended linearly.
\[
    \begin{split}
        \Hmg_\lambda: \C^\lambda[x_1,\ldots,x_m] &\rightarrow V(\lambda)\\
        x^{\mu} &\mapsto x^{\mu}y^{\lambda-\mu}
    \end{split}
\]
$\C^n[x]$ and $V(n)$ are isomorphic as vector spaces via $\Hmg_n$, and we will mainly utilize bivariate homogeneous polynomials in $V(n)$ over the usual univariate polynomials in $\C^n[x]$.
%
%of course is the fact that elements of $V(n)$ are polynomials in two variables, and hence have many more zeros than elements of $\C^n[x]$. By homogeneity though, those roots have strong structural properties which essentially mimic the root properties of elements of $\C^n[x]$.
%
What homogeneity gets us is a simplification of a number of issues related to the fact that polynomials in $\C^n[x]$ have \emph{at most} $n$ zeros. Specifically, it is more natural to think of the missing zeros (when the number of zeros is less than $n$) as being ``at infinity''. Certain results require premises restricting to convex regions or to polynomials of degree exactly $n$ (e.g., the classical Grace-Walsh-Szeg{\H{o}} coincidence theorem), and such details vanish when considering homogeneous polynomials with possible roots at infinity. Another way to say this is that we consider polynomials in $V(n)$ to have exactly $n$ roots in $\CP^1$, which can also be thought of as the Riemann sphere. We also consider polynomials $p(x,y) = p(x_1,y_1,\ldots,x_m,y_m) \in V(\lambda)$ to have zeros in $(\CP^1)^m$, where each pair $(x_k,y_k)$ corresponds to a single factor of $\CP^1$ in $(\CP^1)^m$.

    We use the notation $(a:b) \in \CP^1$, which is meant to give off the connotation of a ratio; that is, $(a:b)$ should feel like $a/b$. We also use the notation $(a:b) = \big((a_1:b_1), \ldots, (a_m:b_m)\big) \in (\CP^1)^m$. Note that this connotation aligns with the idea of considering the zeros of polynomials to be in $(\CP^1)^m$. given the following equality. Defining $p := \Hmg_\lambda(f)$ for a given polynomial $f \in \C^\lambda[x_1,\ldots,x_m]$, we have:
    \[
        p(x,y) = p(x_1,y_1,\ldots,x_m,y_m) = \prod_k y_k^{\lambda_k} \cdot f(x_1/y_1,\ldots,x_m/y_m) = y^\lambda \cdot f(x/y)
    \]
%\end{remark}
%
% For elements of $\CP^1$ and for homogeneous pairs of variables, we use the notation $(a;b) \in \CP^1$ and $p(x;y) = \sum_{k=0}^n p_k x^k y^{n-k} \in V(n)$. Similarly, we use the notation $(a;b) = ((a_1;b_1), \ldots, (a_m;b_m)) \in (\CP^1)^m$
%
Finally, we give an important definition which will essentially replace the notion of a monic polynomial for homogeneous polynomials.

\begin{definition}
    Given $p \in V(\lambda)$, we say that $p$ is \emph{top-degree monic} if the coefficient of $x^\lambda$ in $p$ equals 1. In particular, if $p \in V(n)$ is top-degree monic, then $p$ has no roots at infinity.
\end{definition}

\section{Homogeneous Polynomials as Representations}\label{reps_sect}

In this section, we will discuss some basic representation theory of $\SL_2(\C)$ and show how the apolarity form and the notion of the symbol of an operator arise naturally in the representation theoretic context. Most of the representation theory we use in this section is very basic. There are a number of references which discuss the theory in full detail, albeit with different goals in mind. Typically this is done via the theory of Lie groups and algebras, as in \cite{fulton2013reptheory} and in \cite{humphreys2012lietheory}.

As a note, most of the content of this section is less relevant to the analytic questions associated to polynomials. Rather, it serves as the foundational structure for a new approach to Grace's theorem and results concerning stability-preserving operators. For this reason, we believe it worthwhile to explicate key aspects of this foundation and their connection to analytic results. Pushing further into this connection may lead to new results beyond the scope of this paper.

\subsection{The Action of \texorpdfstring{$\SL_2(\C)$}{SL2(C)}}\label{SL_action_subsect}

Given $(\alpha:\beta) \in \CP^1$ and $\phi \in \SL_2(\C)$, we define the usual action of $\SL_2(\C)$ on $\CP^1$ via $\phi \cdot (\alpha:\beta) := \phi \binom{\alpha}{\beta}$. That is, $\phi$ acts by matrix multiplication on the vector $\binom{\alpha}{\beta}$. Equivalently, $\phi$ acts as its corresponding M{\"o}bius transformation on $\CP^1$. Note that, as with M{\"o}bius transformations, $\SL_2(\C)$ acts transitively on circular regions in $\CP^1$.

This then induces an action on $V(n)$ by acting on the roots (in $\CP^1$) of polynomials in $V(n)$. Given $p \in V(n)$ and $\phi \in \SL_2(\C)$, this action is defined via:
\[
    (\phi \cdot p)(x,y) := p\left(\phi^{-1} \binom{x}{y}\right) \equiv (p \circ \phi^{-1})(x,y)
\]
We can define a similar action of $(\SL_2(\C))^m$ on $V(\lambda)$, for $\lambda \in \N_0^m$. Specifically, given $p \in V(\lambda)$ and $(\phi_1,\ldots,\phi_m) \in \SL_2(\C)^m$, this action is defined via:
\[
    \big((\phi_1,\ldots,\phi_m) \cdot p\big)(x_1,y_1,\ldots,x_m,y_m) := p\left(\phi_1^{-1} \binom{x_1}{y_1}, \ldots, \phi_m^{-1}\binom{x_m}{y_m}\right)
\]
These actions turn $V(n)$ and $V(\lambda)$ into representations of $\SL_2(\C)$ and $(\SL_2(\C))^m$, respectively. These are precisely the finite dimensional irreducible representations of $\SL_2(\C)$ and $(\SL_2(\C))^m$ (see Lecture 11 of \cite{fulton2013reptheory}), and so they are the basic building blocks of the $\SL_2(\C)$ representation theory. Actions on $V(n)$ and $V(\lambda)$ can be extended to tensor products in the usual way, and in this paper we will make use of both inner and outer tensor products. We now briefly discuss tensor product actions for those less familiar.

The outer tensor product of $V(\lambda_k)$, denoted $V(\lambda_1) \boxtimes \cdots \boxtimes V(\lambda_m)$, is a representation of $(\SL_2(\C))^m$ with action by $(\phi_1,\ldots,\phi_m)$ on simple tensors given as follows:
\[
    (\phi_1, \cdots, \phi_m) \cdot (p_1 \boxtimes \cdots \boxtimes p_m) := (\phi_1 \cdot p_1) \boxtimes (\phi_2 \cdot p_2) \boxtimes \cdots \boxtimes (\phi_m \cdot p_m)
\]
This implies that $V(\lambda)$ and $V(\lambda_1) \boxtimes \cdots \boxtimes V(\lambda_m)$ are isomorphic as representations, and this fact will be used when we define the symbol in \S\ref{symb_subsect}.

The inner tensor product of $V(\lambda_k)$, denoted $V(\lambda_1) \otimes \cdots \otimes V(\lambda_m)$, is a representation of $\SL_2(\C)$ with action by $\phi$ on simple tensors given as follows:
\[
    \phi \cdot (p_1 \otimes \cdots \otimes p_m) := (\phi \cdot p_1) \otimes (\phi \cdot p_2) \otimes \cdots \otimes (\phi \cdot p_m)
\]
While $V(\lambda)$ and $V(\lambda_1) \otimes \cdots \otimes V(\lambda_m)$ are isomorphic as vector spaces, they are representations of different groups ($(\SL_2(\C))^m$ and $\SL_2(\C)$ respectively). The inner tensor product relates to invariants of multiple polynomials with respect to a single $\SL_2(\C)$ action. For instance, the apolarity form takes two distinct polynomials as input, and it is a classical result that this form is invariant with respect to a single action by M{\"o}bius transformation. As it turns out, this form can be viewed as an $\SL_2(\C)$-invaiant map on an inner tensor product of polynomial spaces. It will therefore be important for us to understand these inner tensor products in a little more detail.

\subsection{An Important Invariant Map, and Apolarity}\label{inv_maps_subsect}

To aide in our investigation of inner tensor products of $\SL_2(\C)$ representations, we now define an important $\SL_2(\C)$-invariant linear map, denoted by $D$. This map has a long history in invariant theory, and we touch on this below.

\begin{proposition}\label{Dmap_prop}
    The linear map $D := (\partial_x \otimes \partial_y - \partial_y \otimes \partial_x): V(n+1) \otimes V(m+1) \rightarrow V(n) \otimes V(m)$ is $\SL_2(\C)$-invariant.
\end{proposition}
\begin{proof}
    It suffices to check this on simple tensors. Fix $\phi = \left[\begin{smallmatrix} a & b \\ c & d \end{smallmatrix}\right] \in \SL_2(\C)$, $p \in V(n+1)$, and $q \in V(m+1)$. We compute:
    \[
        \begin{split}
            (\phi^{-1} \circ D \circ \phi) (p \otimes q) &= (\phi^{-1} \circ D) (p(dx-by, -cx+ay) \otimes q(dx-by, -cx+ay)) \\
                &= (d\partial_x-c\partial_y)p \otimes (-b\partial_x+a\partial_y)q - (-b\partial_x+a\partial_y)p \otimes (d\partial_x-c\partial_y)q \\
                &= (ad-bc)(\partial_x p \otimes \partial_y q - \partial_y p \otimes \partial_x q) \\
                &= D (p \otimes q)
        \end{split}
    \]
    That is, $D \circ \phi = \phi \circ D$.
\end{proof}

\begin{proposition}
    The multiplication map $V(n) \otimes V(m) \xrightarrow{\times} V(n+m)$ is $\SL_2(\C)$-invariant.
\end{proposition}
\begin{proof}
    Trivial.
\end{proof}

Powers of the $D$ map actually appear in the literature under a few different names. The first comes from invariant theory, where the application of the map
\[
    V(n) \otimes V(m) \xrightarrow{D^r} V(n-r) \otimes V(m-r) \xrightarrow{\times} V(n+m-2r)
\]
to polynomials $p \in V(n)$ and $q \in V(m)$ is called the \emph{$r^\text{th}$ transvectant} of $p$ and $q$. This map is also the result of the $r^\text{th}$ iteration of \emph{Cayley's $\Omega$ process}. These notions are discussed, for example, in chapters 4 and 5 of \cite{olver1999invarianttheory}, where they are used to explicitly compute invariants and covariants of forms. In particular, the invariance of the Jacobian ($1^\text{st}$ transvectant map applied to $p \otimes q$) and the Hessian ($2^\text{nd}$ transvectant map applied to $p \otimes p$) can be determined in this way.

Additionally, the $n^\text{th}$ transvectant of $p,q \in V(n)$ is used to define a notion of apolarity (see, e.g., \cite{rota1993apolarity} and \cite{brennan2007apolarity}), and this notion corresponds to the classical one used in Grace's theorem. In fact, one of the original formulations of Grace's theorem can be found in Grace and Young's 1903 book, \emph{The Algebra of Invariants} \cite{grace1903invariants}. This suggests a connection between invariant theory and the analytic consequences of apolarity theory via the $D$ map, and we will indeed see this map play a crucial role in the proof of Grace's theorem (Theorem \ref{graces_thm}).

We are now ready to define the homogeneous apolarity form via the $D$ map. This form and its $\SL_2(\C)$-invariance are then the next main step toward the definition of the symbol of an operator. In the next section, we will use this bilinear form to define an important construction called the dual of a representation. This will serve as the link to viewing spaces of linear operators as representations themselves.

\begin{definition}\label{Uform_def}
    We call the $n^\text{th}$ transvectant
    \[
        V(n) \otimes V(n) \xrightarrow{D^n} V(0) \otimes V(0) \cong \C
    \]
    the \emph{apolarity form} of $V(n)$. This is the unique (up to scalar) nondegenerate $\SL_2(\C)$-invariant bilinear form on $V(n)$, and therefore it is the homogenization of the classical apolarity form.
\end{definition}

We now want to extend this definition to act on $V(\lambda) \otimes V(\lambda)$ for $\lambda \in \N_0^m$. Note that for $p \in V(\lambda)$, we have $m$ pairs of variables given by $p(x_1,y_1,\ldots,x_m,y_m)$, which allows us to naturally define:
\[
    D^\lambda := \prod_{i=1}^m (\partial_{x_i} \otimes \partial_{y_i} - \partial_{y_i} \otimes \partial_{x_i})^{\lambda_i}
\]
With this, we can define the apolarity form for $V(\lambda)$ as follows.

\begin{definition}\label{Mform_def}
    We call the map
    \[
        V(\lambda) \otimes V(\lambda) \xrightarrow{D^\lambda} V(0^m) \otimes V(0^m) \cong \C
    \]
    the \emph{apolarity form} of $V(\lambda)$. This is the unique (up to scalar) nondegenerate $(\SL_2(\C))^m$-invariant bilinear form on $V(\lambda)$, and therefore it is the homogenization of the multivariate apolarity form defined by Borcea and Br{\"a}nd{\'e}n in \cite{bb2}.
\end{definition}

Since $V(0) \otimes V(0) \cong \C$ and $V(0^m) \otimes V(0^m) \cong \C$, we will often consider the maps $D^n$ and $D^\lambda$ to output a element of $\C$. And as a final note, we do not justify here the claims of uniqueness and nondegeneracy stated above. Proving these claims involves decomposing $V(n) \otimes V(n)$ and $V(\lambda) \otimes V(\lambda)$ into their irreducible components, and we leave this work to Appendix \ref{sl2_app} for the interested reader (see Corollaries \ref{Uunique_cor} and \ref{Munique_cor} specifically).

\subsection{The Symbol of an Operator}\label{symb_subsect}

Given representations $V(\lambda)$ and $V(\alpha)$ (for $\lambda \in \N_0^m$ and $\alpha \in \N_0^l$) of $(\SL_2(\C))^m$ and $(\SL_2(\C))^l$ respectively, the space of linear maps between $V(\lambda)$ and $V(\alpha)$ can be viewed as a representation of $(\SL_2(\C))^{m+l}$ in a standard way. This space of linear maps is denoted $\Hom(V(\lambda),V(\alpha))$. As discussed previously, we will now use the apolarity form defined above to construct a representation isomorphism between $\Hom(V(\lambda),V(\alpha))$ and $V(\lambda \sqcup \alpha)$ (which is a space of polynomials in $m+l$ variables). This will lead us to a natural definition for the symbol of an operator.

The significance of this isomorphism will come from the fact that stability results about $V(\lambda \sqcup \alpha)$ will transfer to $\Hom(V(\lambda),V(\alpha))$ via the symbol lemma (Lemma \ref{symb_lemma}) stated below. We will see in \S\ref{str_Cstab_subsect} that this lemma and Grace's theorem almost immediately imply a characterization of stability-preserving operators which is similar to that of Borcea and Br{\"a}nd{\'e}n.

To this end, consider the standard representation isomorphism $\Hom(V(\lambda),V(\alpha)) \cong V(\lambda)^* \boxtimes V(\alpha)$, given by $T \mapsto \sum_{\mu \leq \lambda} (x^\mu y^{\lambda-\mu})^* \boxtimes T(x^\mu y^{\lambda-\mu})$, where $V(\lambda)^*$ is the dual representation of $V(\lambda)$. We omit here the details regarding explicit definitions of the action of (products of) $\SL_2(\C)$ on $\Hom$ and dual representations. Instead, we utilize the fact that the apolarity form provides an $(\SL_2(\C))^m$-invariant isomorphism between $V(\lambda)$ and the dual representation $V(\lambda)^*$, as stated in the following result.

\begin{proposition}
    For any $\lambda \in \N_0^m$, there is an $(\SL_2(\C))^m$-invariant isomorphism $V(\lambda)^* \rightarrow V(\lambda)$ given by $(x^\mu y^{\lambda-\mu})^* \mapsto \binom{\lambda}{\mu}(-1)^{\mu} x^{\lambda-\mu} y^\mu$.
\end{proposition}
\begin{proof}
    We use the apolarity form to determine the isomorphism. In particular, up to scalar $(x^\mu y^{\lambda-\mu})^*$ maps to an element $p \in V(\lambda)$ such that $(x^\mu y^{\lambda-\mu})^* = D^\lambda(p \otimes \cdot)$. We compute:
    \[
        D^\lambda(p \otimes x^\alpha y^{\lambda-\alpha}) = \alpha!(\lambda-\alpha)! \binom{\lambda}{\alpha} (-1)^{\alpha} \partial_x^{\lambda-\alpha} \partial_y^\alpha p = \lambda! (-1)^{\alpha} \partial_x^{\lambda-\alpha} \partial_y^\alpha p
    \]
    Picking $p(x,y) := (\lambda!)^{-2}\binom{\lambda}{\mu}(-1)^{\mu} x^{\lambda-\mu} y^\mu$ achieves the desired equality exactly, and therefore $(x^\mu y^{\lambda-\mu})^* \mapsto \binom{\lambda}{\mu}(-1)^{\mu} x^{\lambda-\mu} y^\mu$ is an $(\SL_2(\C))^m$-invariant isomorphism.
\end{proof}

With this, we consider the following string of $(\SL_2(\C))^{m+l}$-invariant isomorphisms:
\[
    \Hom(V(\lambda),V(\alpha)) \rightarrow V(\lambda)^* \boxtimes V(\alpha) \rightarrow V(\lambda) \boxtimes V(\alpha) \rightarrow V(\lambda \sqcup \alpha)
\]
The first map is the standard isomorphism discussed above, the second map is induced by the previous proposition, and the third map is given by the discussion of outer tensor products in \S\ref{SL_action_subsect}. This string of maps is explicitly defined on a given linear operator via:
\[
    \begin{split}
        T &\mapsto \sum_{\mu \leq \lambda} (z^\mu w^{\lambda-\mu})^* \boxtimes T(x^\mu y^{\lambda-\mu}) \\
            &\mapsto \sum_{\mu \leq \lambda} \binom{\lambda}{\mu} (-1)^{\mu} z^{\lambda-\mu} w^{\mu} \boxtimes T(x^\mu y^{\lambda-\mu}) \\
            &\mapsto \sum_{\mu \leq \lambda} \binom{\lambda}{\mu} z^{\lambda-\mu} (-w)^{\mu} \cdot T(x^\mu y^{\lambda-\mu})
    \end{split}
\]
Here, $T$ acts only on the $x$ and $y$ variables, and $z$ and $w$ are the $\lambda$ variables in $V(\lambda \sqcup \alpha)$. This gives the desired isomorphism between $\Hom(V(\lambda),V(\alpha))$ and $V(\lambda \sqcup \alpha)$, and hence we refer to this map as the $\Symb$ map.

\begin{definition}\label{symb_def}
    For $\lambda \in \N_0^m$ and $\alpha \in \N_0^l$, we define the following $(\SL_2(\C))^{m+l}$-invariant isomorphism:
    \[
        \Symb: \Hom(V(\lambda),V(\alpha)) \rightarrow V(\lambda \sqcup \alpha)
    \]
    \[
        T \mapsto T\left[(zy-xw)^\lambda\right] = \Hmg_{(\lambda,\alpha)}\left(T[(z-x)^\lambda]\right) = \sum_{\mu \leq \lambda} \binom{\lambda}{\mu} z^{\lambda-\mu} (-w)^{\mu} \cdot T(x^\mu y^{\lambda-\mu})
    \]
    We call $\Symb(T)$ the \emph{(universal) symbol of $T$}. 
\end{definition}

This expression bears striking resemblance to the symbol used by Borcea and Br{\"a}nd{\'e}n in \cite{bb1}, which motivates the use of the name ``symbol'' here. (In fact, $\Symb$ is almost the homogenization of the Borcea-Br{\"a}nd{\'e}n symbol.) In \S\ref{Cops_sect}, $\Symb$ will allow us to reduce the study of $\Hom(V(\lambda),V(\alpha))$ to the study of $V(\lambda \sqcup \alpha)$ via the next lemma. We refer to this next result as the symbol lemma, and it demonstrates the fundamental connection between an operator $T$, its symbol, and the apolarity form. Note that the computation done here in the proof of this lemma is in a sense redundant. The operator $\Symb$ was essentially defined such that $\Symb(T)$ acts as $T$ via $D^\lambda$.

\begin{lemma}[Symbol Lemma]\label{symb_lemma}
    Fix $\lambda \in \N_0^m$, $\alpha \in \N_0^l$, and a linear operator $T \in \Hom(V(\lambda),V(\alpha))$. For $q \in V(\lambda)$ and $r \in V(\alpha)$, we have:
    \[
        D^\lambda(\Symb(T) \otimes q \cdot r) = (\lambda!)^2 T(q) \otimes r
    \]
\end{lemma}
\begin{proof}
    Letting $q_\mu$ be the coefficient of the $x^\mu y^{\lambda-\mu}$ term of $q$, we compute:
    \[
        \begin{split}
            D^\lambda(\Symb(T) \otimes q \cdot r) &= D^\lambda \left(\sum_{\mu \leq \lambda} \binom{\lambda}{\mu} x^{\lambda-\mu} (-y)^{\mu} \cdot T(x^\mu y^{\lambda-\mu}) \otimes q \cdot r\right) \\
                &= \sum_{\mu \leq \lambda} \binom{\lambda}{\mu}^2 (\lambda-\mu)! \mu! \cdot T(x^\mu y^{\lambda-\mu}) \otimes (\partial_x^\mu \partial_y^{\lambda-\mu} q) r \\
                &= (\lambda!)^2 \sum_{\mu \leq \lambda} T(x^\mu y^{\lambda-\mu}) \otimes q_\mu \cdot r \\
                &= (\lambda!)^2 T(q) \otimes r \\
        \end{split}
    \]
\end{proof}

\section{Polynomial Stability Theory}\label{stability_sect}

Given $\lambda \in \N_0^m$, a polynomial $p(x_1,y_1,\ldots,x_m,y_m) \in V(\lambda)$ is said to be \emph{stable} if it doesn't vanish in $\UHP^m \subset (\CP^1)^m$. More generally, $p$ is said to be \emph{$\Omega$-stable} if it doesn't vanish in $\Omega$. As above, we say $p$ is \emph{weakly $\Omega$-stable} if possibly $p \equiv 0$. Most all results related to zero location of polynomials then can be translated into statements about stability properties of polynomials and stability preservation properties of operations applied to polynomials.

A linear operator $T$ is said to \emph{preserve weak $\Omega$-stability} if $T(p)$ is $\Omega$-stable or identically zero for all $\Omega$-stable $p$. Further, a real linear operator $T$ \emph{preserves weak real stability} if the same holds for real stable polynomials. In \cite{bb1}, Borcea and Br{\"a}nd{\'e}n were concerned with classifying such weak stability preserving operators. As seen in their main characterization results (Theorems \ref{BB_Cops_thm} and \ref{BB_Rops_thm}), allowing the zero polynomial leads to a degeneracy condition in their characterization.

In order to remove this condition, we define a slightly different notion of stability: we say a linear operator $T$ \emph{preserves (strong) $\Omega$-stability} if $T(p)$ is stable and nonzero for all stable $p$. Similarly, we say a real linear operator $T$ \emph{preserves (strong) real stability} if the same holds for real stable polynomials. Most of the main results of this paper rely on this notion of strong stability preservation, and we will demonstrate how it relates to weak stability preservation in \S\ref{Cops_deriv_subsect} and \S\ref{Rops_deriv_subsect}.

\subsection{Polar Derivatives}

A crucial tool of classical stability theory is the polar derivative. In particular, this notion leads to Laguerre's theorem (Proposition \ref{laguerre_prop}), which is the main lemma toward Grace's theorem. By passing to homogeneous polynomials the polar derivative becomes conceptually simpler, and this in turn sheds further light on the general connection to $\SL_2(\C)$-invariance and the $D$ map. One example of this, as we will see below, is that the polar derivative can be defined as the conjugation of $\partial_x$ by some $\SL_2(\C)$ action.

Given some ``pole'' $x_0 \in \C$, the \emph{polar derivative} with respect to $x_0$ of $f \in \C^n[x]$ is classically defined as follows.
\[
    (d_{x_0}f)(x) := nf(x) - (x-x_0)f'(x)
\]
Noticing that the term of degree $n$ cancels out, the resulting polynomial is of degree $n-1$. It is typically said that this operator generalizes the ordinary derivative in the sense that $\lim_{x_0 \rightarrow \infty} x_0^{-1}d_{x_0}f(x) = f'(x)$. However, this operator also generalizes the ordinary derivative in more natural way, which we see by passing to $V(n)$.

Fix any $\phi = \begin{bmatrix} a & b \\ c & d \end{bmatrix} \in \SL_2(\C)$. Define the \emph{pole} of $\phi$ to be $(-d:c) \in \CP^1$. For $p \in V(n)$, we then define the polar derivative with respect to $\phi$ as follows.
\[
    d_\phi p := (\phi^{-1}\partial_x\phi)p = -(-d\partial_x + c\partial_y)p
\]
Notice that $d_\phi$ depends only on $\phi^{-1}\binom{-1}{0} = \binom{-d}{c}$. With this, the pole of $\phi$ should be interpreted as the element of $\CP^1$ that $\phi$ sends to $\infty = (-1:0)$.

This definition of the polar derivative with respect to $\phi$ is at very least a natural one, as it can be simply described as the conjugation of $\partial_x$ by the action of $\phi$. The following result then shows that this is actually the correct definition.

\begin{proposition}
Fix $\phi \equiv \begin{bmatrix} a & b \\ c & d \end{bmatrix} \in \SL_2(\C)$ with pole $(-d:c)$, and define $x_0 := \frac{-d}{c}$ (for $c \neq 0$). Then:
\[
    d_\phi \circ \Hmg_n = \Hmg_{n-1} \circ (-c \cdot d_{x_0})
\]
That is, the polar derivative $d_\phi$ on $V(n)$ is the homogenization of the classical polar derivative $d_{x_0}$ on $\C^n[x]$ (up to scalar).
\end{proposition}
\begin{proof}
    Straightforward computation.
\end{proof}

As mentioned above, $d_\phi$ depends only on $(-d:c)$, the pole of $\phi$. So given any pole in $\CP^1$, we can actually choose $\phi \in \SL_2(\C)$ to be a rotation of the Riemann sphere (i.e., $\CP^1$). This then gives the following intuitive description of the polar derivative.

\begin{remark}
    Fix a rotation $\phi \in \SL_2(\C)$ with pole $(-d:c)$. The polar derivative $d_\phi$ then acts on $p \in V(n)$ in the following way. First, consider the zeros of $p$ as being placed in the Riemann sphere via stereographic projection. Next, rotate the sphere via $\phi$, which moves $(-d:c)$ to infinity at the top of the sphere. Apply the derivative to the new polynomial given by the new locations of the zeros. Finally, undo the original rotation via $\phi^{-1}$, which moves infinity back to the pole $(-d:c)$.
\end{remark}

\subsection{Projective Convexity and Laguerre's Theorem}\label{proj_conv_subsect}

Circular regions play a key role in Grace's theorem and its corollaries. The main reason for this is Laguerre's theorem, which essentially says that polar derivatives with respect to points of a circular region preserve stability for that circular region. This theorem in turn relies on the Gauss-Lucas theorem, which deals with convex regions. 

A \emph{circular region} in $\C$ is defined to be a disc, half-plane, or complement of a disc, and such a circular region can be either open or closed. The generalization of circular regions to $\CP^1$ is the obvious one. A \emph{circular region} in $\CP^1$ is defined to be the sets in $\CP^1$ for which the stereographic projection is a circular region in $\C$. Note that $\SL_2(\C)$ acts transitively on the set of all circular regions in $\C$ or in $\CP^1$. We now state a lemma to Laguerre's theorem, which gets at the heart of the importance of circular regions.

\begin{lemma}
    Let $C \subseteq \CP^1$ be a circular region, and let $\phi \in \SL_2(\C)$ be such that its pole is not in $C$. Then, the stereographic projection of $\phi \cdot C$ is convex.
\end{lemma}
\begin{proof}
    Let $(x_0:y_0) \notin C$ be the pole of $\phi$. Then, $\phi$ maps $(x_0:y_0)$ to $\infty \in \CP^1$ and maps $C$ to another circular region. Since $(x_0:y_0) \notin C$ implies $\infty \notin \phi \cdot C$, the sterographic projection of $\phi \cdot C$ is either an open half-plane or is bounded away from $\infty$. Since $\phi \cdot C$ is a circular region, it must be convex.
\end{proof}

This then leads to a natural extension of the notion of a circular region.

\begin{definition}
    Given $C \subseteq \CP^1$, we say that $C$ is \emph{projectively convex} if for every $\phi \in \SL_2(\C)$ with pole not in $C$, the stereographic projection of $\phi \cdot C$ is convex.
\end{definition}

We now classify all projectively convex sets in $\CP^1$ in the following. This result has been demonstrated before in \cite{zervos1960aspects}, where projectively convex regions are referred to as generalized circular regions.

\begin{proposition}[Zervos] \label{proj_conv_characterization}
    Let $C \subseteq \CP^1$ be projectively convex. Then, $C = C^\circ \cup \gamma$, where $C^\circ$ is an open circular region which is the interior of $C$, and $\gamma$ is a connected subset of the boundary of $C^\circ$. In particular, projective convexity is preserved under taking complements.
\end{proposition}

So, one example of a projectively convex set which is not quite a circular region is $\UHP \cup \R_+$. Another is $\UHP \cup [0,1]$. Yet another (albeit after a bit of consideration) is $\UHP \cup (-\infty,0) \cup (1,\infty]$. We now state a homogeneous version of Laguerre's theorem, extended to projectively convex sets.

\begin{proposition}[Laguerre]\label{laguerre_prop}
    Let $C \subseteq \CP^1$ be projectively convex, and fix $\phi \in \SL_2(\C)$. If the pole of $\phi$ is in $C$, then $d_\phi$ preserves strong $C$-stability.
\end{proposition}
\begin{proof}
    Gauss-Lucas and the fact that $\CP^1 \setminus C$ is projectively convex give the result. Specifically, for $C$-stable $p \in V(n)$ consider $\phi \cdot p$, which is stable in $\phi \cdot C \ni \infty$. Letting $B$ be the complement of $C$, the dehomogenization of this polynomial is then of degree exactly $n$ with all of its roots in the stereographic projection of $\phi \cdot B$. By projective convexity, $\phi \cdot B$ is convex and therefore Gauss-Lucas implies $\partial_x(\phi \cdot p)$ is $\phi \cdot C$-stable and not identically zero. Applying $\phi^{-1}$ then implies $d_\phi p = (\phi^{-1} \partial_x \phi) p$ is $C$-stable.
\end{proof}

\begin{corollary}\label{laguerre_cor}
    Let $C_k \subseteq \CP^1$ be projectively convex regions for $k \in [m]$, and fix $\phi \in \SL_2(\C)$. If the pole of $\phi$ is in $C_{k_0}$, then $d_\phi$ acting on the variables $(x_{k_0},y_{k_0})$ preserves strong $(C_1 \times \cdots \times C_m)$-stability.
\end{corollary}
\begin{proof}
    Follows from the fact that taking derivatives in some variables commutes with evaluation in the others. Specifically, $p \in V(\lambda)$ is $(C_1 \times \cdots \times C_m)$-stable iff $p \neq 0$ for all evaluations in $C_1 \times \cdots \times C_m$. So, evaluating $p$ in all variables in that product of sets except $(x_{k_0},y_{k_0})$ gives us a $C_{k_0}$-stable polynomial in $V(\lambda_{k_0})$. Applying the previous proposition then gives the result.
\end{proof}

\subsection{Real Stable Polynomials}\label{stab_subsect}

We now give a number of classical real stability results, along with a few results from \cite{bb1} and \cite{bb2}. Additionally, we state these results for homogeneous polynomials in $V_\R(\lambda)$, taking roots at infinity into account. The results of this section will come in to play mainly in \S\ref{Rops_sect}, where we discuss real linear operators and operators preserving interval- and ray-rootedness.

The first result we will need for our considerations of $V_\R(\lambda)$ is a version of the Hermite-Biehler theorem, often called the Hermite-Kakeya-Obreschkoff theorem. We state here without proof the multivariate version essentially used in Theorem 1.9 of \cite{bb1} (see also \S2.4 of \cite{wagner2011polys}). First we need a definition.

\begin{definition}\label{pp_def}
    We say that $p,q \in V_\R(\lambda)$ are in \emph{proper position}, denoted $p \ll q$, if $q + ip$ is weakly $\UHP^m$-stable (equivalently, if $p + iq$ is weakly $\LHP^m$-stable).
\end{definition}

\begin{proposition}[Multivariate Hermite-Biehler] \label{hermite-biehler}
    For $p,q \in V_\R(\lambda)$, $ap + bq$ is weakly real stable for all $a,b \in \R$ if and only if either $p \ll q$ or $q \ll p$.
\end{proposition}

This result will be crucial to our consideration of real polynomials and real stability (as it was in \cite{bb1}). Its main use for us in this direction is made explicit in the following.

\begin{lemma}\label{rs_lemma}
    Fix $\lambda \in \N_0^m$, $\alpha \in \N_0^l$, and a linear operator $T: V(\lambda) \to V(\alpha)$ which restricts to a real linear operator from $V_\R(\lambda)$ to $V_\R(\alpha)$. If $T$ preserves weak real stability and $p \in V(\lambda)$ is stable, then $T(p)$ is either $\UHP^l$-stable, $\LHP^l$-stable, or identically zero.
\end{lemma}
\begin{proof}
    By the Hermite-Biehler theorem, there exist $q,r \in V_\R(\lambda)$ such that $p = q + ir$ and $a q + b r$ is real stable or zero for all $a,b \in \R$. So, $a T(q) + b T(r)$ is real stable or zero for all $a,b \in \R$. By Hermite-Biehler again, $T(p) = T(q) + iT(r)$ is either $\UHP^l$-stable, $\LHP^l$-stable, or identically zero.
\end{proof}

The next two results are from \cite{bb1}, the first of which gives an equivalent characterization for a polynomial to be a scalar multiple of a real stable polynomial. This result will be specifically used in \S\ref{Rops_sect} to generalize complex operator theoretic stability results to the real stability case.

\begin{lemma}[\cite{bb1}, Proposition 4.1]\label{scal_rs_lemma}
    Let $p \in V_\C(\lambda)$ be both $\UHP^m$-stable and $\LHP^m$-stable. Then, $p$ is a (complex) scalar multiple of a real stable polynomial. In particular, if nonzero $q,r \in V_\R(\lambda)$ are such that $q \ll r$ and $r \ll q$, then $r$ is a (real) scalar multiple of $q$.
\end{lemma}

The next result provides the degeneracy cases in the Borcea-Br{\"a}nd{\'e}n characterizations (recall the dimension restrictions of Theorems \ref{BB_Cops_thm} and \ref{BB_Rops_thm}). We will use this result to explicate the link between our operator characterization and the Borcea-Br{\"a}nd{\'e}n characterization (see Lemmas \ref{bb_Clink} and \ref{bb_Rlink}).

\begin{lemma}[\cite{bb1}, Lemma 3.2]\label{dim_lemma}
    Let $W \subseteq V_\K(\lambda)$ be a $\K$-vector subspace (for $\K = \C$ or $\K = \R$) consisting only of weakly stable (resp. weakly real stable) polynomials. We have:
    \begin{enumerate}[label=(\alph*)]
        \item If $\K = \C$, then $\dim(W) \leq 1$.
        \item If $\K = \R$, then $\dim(W) \leq 2$.
    \end{enumerate}
\end{lemma}

By applying appropriate M{\"o}bius transformations, note that $(a)$ of the above lemma can be generalized to $(C_1 \times \cdots \times C_m)$-stable polynomials for any open circular regions $C_1,\ldots,C_m \subseteq \CP^1$.

We now state the last result of this section, which refines the Hermite-Biehler theorem for top-degree monic polynomials in $V_\R(n)$. This refinement comes through the notion of \emph{interlacing polynomials} and is much closer to the original statement of the classical Hermite-Biehler theorem (e.g., see Theorem 6.3.4 in \cite{rahman2002anthpoly}).

\begin{lemma}\label{total_order_lemma}
    For top-degree monic $p,q \in V_\R(n)$, $p \ll q$ if and only if the roots of $p$ and $q$ (denoted in increasing order by $(\alpha_k:1)$ and $(\beta_k:1)$, respectively) interlace on the real line in the following way:
    \[
        \alpha_1 \leq \beta_1 \leq \alpha_2 \leq \beta_2 \leq \cdots \leq \alpha_n \leq \beta_n
    \]
    Further, if these equivalent conditions hold, then $\ll$ gives a total order on the top-degree monic elements of the span of $p$ and $q$ in $V_\R(n)$. This order is equivalently defined via the order of the $k^\text{th}$ largest roots, for any $k \in [n]$ such that $\alpha_k \neq \beta_k$.
\end{lemma}
\begin{proof}
    The fact that $p \ll q$ is equivalent to interlacing roots is the classical univariate Hermite-Biehler theorem. That $q$ has larger roots than $p$ can be obtained by the fact that the $(n-1)^\text{st}$ derivative of $q + ip$ must be $\UHP$-stable. Since both polynomials are top-degree monic, this $(n-1)^\text{st}$ derivative will be a complex linear combination of two linear terms. This complex linear combination is given as follows, where $s_q$ and $s_p$ denote the respective sums of the roots of $q$ and $p$:
    \[
        \partial_x^{n-1} \big(q(x,y) + ip(x,y)\big) = (n! \cdot x - (n-1)! \cdot s_q y) + i(n! \cdot x - (n-1)! \cdot s_p y) = n!(1+i)\left(x - \frac{s_q + i s_p}{n(1 + i)} y\right)
    \]
    Since this polynomial is $\UHP$-stable, it must be that $\frac{s_q + i s_p}{n(1 + i)} \in \overline{\LHP}$. We further compute:
    \[
        \overline{\LHP} \ni \frac{s_q + i s_p}{n(1 + i)} = \frac{(s_q + i s_p)(1-i)}{2n} = \frac{s_q + s_p + i(s_p-s_q)}{2n}
    \]
    Therefore $s_q \geq s_p$, which is the same as saying that the sum of the roots of $q$ is larger than that of $p$. Since we already know that the roots of $q$ and $p$ interlace, this implies that $q$ has larger roots than $p$.
    
    As for the total ordering property, let $r$ and $s$ be two polynomials in the real span of $p$ and $q$. Any real linear combination of these polynomials is then a real linear combination of $p$ and $q$ (and hence is real-rooted), and Hermite-Biehler implies either $r \ll s$ or $s \ll r$. By the above interlacing condition, it is straightforward to see that this total order is given by looking at the order of the $k^\text{th}$ roots, for any $k \in [n]$.
\end{proof}

\section{Grace's Theorem}\label{graces_sect}

We now prove the multivariate homogeneous Grace's theorem for some specific projectively convex regions and then derive a few important corollaries. These corollaries will be almost immediate once Grace's theorem has been proven, and yet will quickly yield stronger results regarding linear operators in the next section.

In the usual proof of the classical univariate Grace's theorem, reference to linear factors of $f \in \C^n[x]$ is necessary. This makes generalization to $\C^\lambda[x_1,\ldots,x_m]$ difficult, as multivariate polynomials do not necessarily have any linear factors. In our new proof, we are able avoid reference to linear terms by using particular features of the $D$ map. This means that our proof method works for any $\lambda$.

\begin{theorem}\label{graces_thm}
    Fix $\lambda \in \N_0^m$ and $p,q \in V(\lambda)$. Also, denote $C := \UHP \cup \overline{\R_+}$ and $\widetilde{C} := \LHP \cup \overline{\R_-}$, where the closures are considered to be in $\CP^1$. If $p$ is $C^m$-stable and $q$ is $\widetilde{C}^m$-stable, then $D^\lambda(p \otimes q) \neq 0$.
\end{theorem}
\begin{proof}
    We prove the theorem by induction on degree. For $\lambda \equiv 0$, the result is obvious. For $|\lambda| \geq 1$, we can assume WLOG that $\lambda_1 \geq 1$ by permuting the variables. Define $\delta_1 := (1,0,0,\ldots,0) \in \N_0^m$.
    
    Since $C$ and $\widetilde{C}$ are projectively convex, Corollary \ref{laguerre_cor} implies $(a\partial_{x_1} + b\partial_{y_1})p$ is $C^m$-stable for all $(a:b) \in C$ and $(c\partial_{x_1} + d\partial_{y_1})q$ is $\widetilde{C}^m$-stable for all $(c:d) \in \widetilde{C}$. To obtain a contradiction, we assume $D^\lambda(p \otimes q) = 0$. For $\alpha \in \UHP \cup \R_+ \subset C$ (equivalently, $-\alpha \in \LHP \cup \R_- \subset \widetilde{C}$), this gives:
    \[
        \begin{split}
            D^{\lambda-\delta_1}\big((\alpha \partial_{x_1} + \partial_{y_1})p \otimes (\alpha \partial_{x_1} - \partial_{y_1})q\big) &= \alpha^2 D^{\lambda-\delta_1}(\partial_{x_1} p \otimes \partial_{x_1} q) - D^{\lambda-\delta_1}(\partial_{y_1} p \otimes \partial_{y_1} q) - \alpha D^\lambda(p \otimes q) \\
                &= \alpha^2 D^{\lambda-\delta_1}(\partial_{x_1} p \otimes \partial_{x_1} q) - D^{\lambda-\delta_1}(\partial_{y_1} p \otimes \partial_{y_1} q)
        \end{split}
    \]
    By induction and the stability properties discussed above, we have $D^{\lambda-\delta_1}(\partial_{x_1} p \otimes \partial_{x_1} q) \neq 0$, $D^{\lambda-\delta_1}(\partial_{y_1} p \otimes \partial_{y_1} q) \neq 0$, and $D^{\lambda-\delta_1}\big((\alpha \partial_{x_1} + \partial_{y_1})p \otimes (\alpha \partial_{x_1} - \partial_{y_1})q\big) \neq 0$. This implies:
    \[
        \alpha^2 D^{\lambda-\delta_1}(\partial_{x_1} p \otimes \partial_{x_1} q) - D^{\lambda-\delta_1}(\partial_{y_1} p \otimes \partial_{y_1} q) \neq 0 \Longrightarrow \alpha^2 \neq \frac{D^{\lambda-\delta_1}(\partial_{y_1} p \otimes \partial_{y_1} q)}{D^{\lambda-\delta_1}(\partial_{x_1} p \otimes \partial_{x_1} q)} \in \C \setminus \{0\}
    \]
    However, we can pick $\alpha \in \UHP \cup \R_+$ such that $\alpha^2$ is any value of $\C \setminus \{0\}$ we want, including that of $\frac{D^{\lambda-\delta_1}(\partial_{y_1} p \otimes \partial_{y_1} q)}{D^{\lambda-\delta_1}(\partial_{x_1} p \otimes \partial_{x_1} q)}$. This contradiction gives the result.
\end{proof}

\subsection{Other Regions}

We now generalize the above theorem to other regions via $\SL_2(\C)$ action and topological considerations. Theorem \ref{grace_pairs_thm} can then be considered our most general form of Grace's theorem. First though, we define two new notions in order to simplify the rest of this section.

\begin{definition} \label{grace_pair_def}
    Fix $m \in \N_0$ and any sets $S_1,S_2 \subseteq (\CP^1)^m$. We call $(S_1,S_2)$ a \emph{Grace pair} if: for all $\lambda \in \N_0^m$ and $p,q \in V(\lambda)$ such that $p$ is $S_1$-stable and $q$ is $S_2$-stable, we have that $D^\lambda(p \otimes q) \neq 0$. That is, if Grace's theorem holds for $S_1$ and $S_2$.
\end{definition}

\begin{definition}
    We say that a Grace pair is \emph{disjoint} if it is of the form $(C_1 \times \cdots \times C_m, B_1 \times \cdots \times B_m)$ and $C_k$ and $B_k$ are disjoint for all $k \in [m]$.
\end{definition}

This yields the following restatement of the above theorem.

\begin{corollary}
    For any $m \in \N_0$, $((\UHP \cup \overline{\R_+})^m, (\LHP \cup \overline{\R_-})^m)$ is a Grace pair.
\end{corollary}

The sets considered above intersect at 2 points ($0$ and $\infty$), and this ends up being crucial to the proof. So, in order to extend to the full generality of Grace's theorem, we will need to find such points even when the stability sets of two polynomials $p$ and $q$ do not a priori intersect at all. To this end, we give the following lemmas.

\begin{lemma} \label{stability_closed_to_open}
    Fix $\lambda \in \N_0^m$ and any closed circular regions $C_1,\ldots,C_m \subset \CP^1$. Let $p \in V(\lambda)$ be $(C_1 \times \cdots \times C_m)$-stable. There exist open circular regions $U_1,\ldots,U_m$ such that $C_k \subset U_k$ for all $k \in [m]$ and $p$ is $(U_1 \times \cdots \times U_m)$-stable.
\end{lemma}
\begin{proof}
    Follows from compactness of $\CP^1$ and closedness of $C_1 \times \ldots \times C_m$ and of the zero set of $p$.
\end{proof}

For the next lemma, note that the boundary of any circular region $C$ is topologically equivalent to the unit circle in $\C$ (i.e., the boundary of the unit disc). With this, we call a portion of the boundary of $C$ \emph{open} if it is open when considered as a subset of the unit circle. Further recall the characterization of projectively convex regions given by Proposition \ref{proj_conv_characterization}.

\begin{lemma} \label{proj_conv_open_boundary}
    Fix $n \in \N_0$ and any projectively convex $C \equiv C^\circ \cup \gamma \subset \CP^1$, where $C^\circ$ is an open circular region and $\gamma$ is a connected portion of its boundary. Let $p \in V(n)$ be $C$-stable. There is an open connected subset $\Gamma$ of the boundary of $C^\circ$ such that $\gamma \subseteq \Gamma$ and $p$ is $(C^\circ \cup \Gamma)$-stable.
\end{lemma}
\begin{proof}
    Let $\partial C^\circ$ denote the boundary of (the closure of) $C^\circ$, and let $S \subseteq \CP^1$ be the intersection of $\partial C^\circ$ and the zero set of $p$. Since the zero set of $p$ is closed, we have that $S$ is closed in $\partial C^\circ$. And further, $\gamma \cap S = \varnothing$ by assumption. Defining $\Gamma$ to be the connected component of $\partial C^\circ \setminus S$ containing $\gamma$ then gives the result.
\end{proof}

Using these lemmas and the $\SL_2(\C)$-invariance of the apolarity form, we obtain the following generalization of Grace's theorem. Here, $(ii)$ and $(iii)$ give the multivariate Grace's theorem proven in \cite{bb2}.

\begin{theorem}\label{grace_pairs_thm}
    For $m \in \N_0$ and $C_1,\ldots,C_m,B_1,\ldots,B_m \subseteq \CP^1$, we have that $(C_1 \times \cdots \times C_m, B_1 \times \cdots \times B_m)$ is a Grace pair for the following regions.
    \begin{enumerate}[label=(\roman*)]
        \item For all $k \in [m]$, $C_k$ and $B_k$ are projectively convex, $C_k \cup B_k = \CP^1$, and $C_k \cap B_k$ is exactly two points.
        \item For all $k \in [m]$, $C_k$ is a closed circular region, $B_k$ is an open circular region, and $C_k \cup B_k = \CP^1$.
        \item For all $k \in [m]$, $C_k$ is an open circular region, $B_k$ is a closed circular region, and $C_k \cup B_k = \CP^1$.
        \item For $m = 1$, $C_1$ and $B_1$ are projectively convex and $C_1 \cup B_1 = \CP^1$.
    \end{enumerate}
\end{theorem}
\begin{proof}
    $(i)$. By Proposition \ref{proj_conv_characterization}, every projectively convex region in $\CP^1$ is the union of an open circular region and a portion of its boundary. Since $C_k \cup B_k = \CP^1$ and $C_k \cap B_k$ is exactly two points, we then must have that $C_k = \phi_k \cdot (\UHP \cup \overline{\R_+})$ and $B_k = \phi_k \cdot (\LHP \cup \overline{\R_-})$ for some $\phi_k \in \SL_2(\C)$. Since $D^\lambda$ is $(\SL_2(\C))^m$-invariant, the result follows from Theorem \ref{graces_thm}.
    
    $(ii)$. Fix $p,q \in V(\lambda)$. If $p$ is $(C_1 \times \cdots \times C_m)$-stable and $q$ is $(B_1 \times \cdots \times B_m)$-stable, then Lemma \ref{stability_closed_to_open} implies $p$ is $(U_1 \times \cdots \times U_m)$-stable for some open circular regions $U_1,\ldots,U_m$ such that $C_k \subset U_k$ for all $k \in [m]$. Since $C_k \cup B_k = \CP^1$, we then have that $U_k \cap B_k$ is open and nonempty. Since $U_k$ and $B_k$ are circular regions, their intersection in fact contains an open annulus or open strip in $\C$. Therefore we may slightly shrink $U_k$ and $B_k$ to get closed circular regions $U_k'$ and $B_k'$ such that $U_k' \cup B_k' = \CP^1$ and $U_k' \cap B_k' = \partial U_k' = \partial B_k'$, where $\partial B_k'$ denotes the boundary of $B_k'$. We can then further remove portions of the respective boundaries of $U_k'$ and $B_k'$ to get projectively convex regions $U_k''$ and $B_k''$ such that $U_k'' \cup B_k'' = \CP^1$ and $U_k' \cap B_k'$ is exactly two points. Since $U_k'' \subset U_k$ and $B_k'' \subset B_k$, we have that $p$ is $(U_1'' \times \cdots \times U_m'')$-stable and $q$ is $(B_1'' \times \cdots \times B_m'')$-stable. Therefore $(i)$ implies $D^\lambda(p \otimes q) \neq 0$, and this implies $(ii)$.

    $(iii)$. Same argument as $(ii)$.
    
    $(iv)$. Let $p,q \in V(n)$ be such that $p$ is $C_1$-stable and $q$ is $B_1$-stable. Defining $B_1' := \CP^1 \setminus C_1 \subseteq B_1$, we further have that $q$ is $B_1'$-stable. So WLOG we may assume that $B_1 = B_1'$. Note that this implies $\partial C_1^\circ = \partial B_1^\circ$; that is, the boundaries coincide. If $C_1$ is a circular region then so is $B_1$, and therefore $D^n(p \otimes q) \neq 0$ by $(ii)$ or $(iii)$. This implies $(iv)$ in this case.
    
    Otherwise by Proposition \ref{proj_conv_characterization}, we have that $C_1 = C_1^\circ \cup \gamma_1$ where $C_1^\circ$ is an open circular region and $\varnothing \neq \gamma_1 \subseteq \partial C_1^\circ$. Analogously we have $B_1 = B_1^\circ \cup \psi_1$ with $\varnothing \neq \psi_1 \subseteq \partial B_1^\circ = \partial C_1^\circ$. Lemma \ref{proj_conv_open_boundary} then implies there exist $\Gamma_1$ and $\Psi_1$, which are open portions of the boundary of $C_1^\circ$, such that $\gamma_1 \subseteq \Gamma_1$, $\psi_1 \subseteq \Psi_1$, $p$ is $(C_1^\circ \cup \Gamma_1)$-stable, and $q$ is $(B_1^\circ \cup \Psi_1)$-stable. Since $\Gamma_1 \cup \Psi_1 = \partial C_1^\circ$, we then further can find closed subsets $\Gamma_1' \subset \Gamma_1$ and $\Psi_1' \subset \Psi_1$ such that $\Gamma_1' \cup \Psi_1' = \partial C_1^\circ$ and $\Gamma' \cap \Psi'$ is exactly two points. Therefore $D^n(p \otimes q) \neq 0$ by $(i)$, and this implies $(iv)$.
\end{proof}

Notice that $(ii)$ and $(iii)$ in this result do not allow for mixed open and closed stability regions. That is, all of the $C_k$ must be open and all of the $B_k$ closed, or vice versa. We show that this particular point cannot be ignored, using the following example.

\begin{example}
    Let $\lambda = (1,1,1)$, denote $E := \CP^1 \setminus \overline{\D}$, and consider the polynomial $p := x_1x_2x_3 - y_1y_2y_3 = \Hmg_\lambda(x_1x_2x_3 - 1)$. First, it is easy to see that $D^\lambda(p \otimes p) = 0$. Also, $p$ is $\D^3$-stable and $E^3$-stable, but it is not $\overline{\D}^3$-stable nor $\overline{E}^3$-stable (zero at $(x_k,y_k) = (1,1)$ for $k \in [3]$). That is, the fact that $D^\lambda(p \otimes p) = 0$ does not contradict $(ii)$ or $(iii)$ of the previous theorem.
    
    On the other hand, $p$ is both $(\overline{\D} \times \D \times \D)$-stable and $(E \times \overline{E} \times \overline{E})$-stable. This shows that $(\overline{\D} \times \D \times \D, E \times \overline{E} \times \overline{E})$ is not a Grace pair. That is, mixed open and closed stability regions cannot be included in $(ii)$ and $(iii)$ of the previous theorem.
\end{example}

As for whether or not the two-point intersection condition can be removed from $(i)$ seems to be a more subtle point. It would be quite nice if this condition could be removed, but it is unclear whether or not it is possible.

\subsection{Evaluation Symbols}

One way to interpret the stability properties of a given polynomial is via the stability-preservation properties of a particular type of linear operator: the evaluation map. That is, the map which evaluates a polynomial $p(x,y)$ at $(a,b) \neq (0,0)$ preserves strong $\{(a:b)\}$-stability, where $\{(a:b)\}$ is a subset of $\CP^1$ consisting of a single point. Further, given $\lambda \in \N_0^m$ and $(a,b) = (a_1,b_1,\ldots,a_m,b_m) \in \C^{2m}$ (with $(a_j,b_j) \neq 0$ for all $j$), we can define the corresponding evaluation map as an element of $\Hom(V(\lambda),V(0))$ since $V(0) \cong \C$. This allows us to obtain symbols for evaluation maps, and these play an important role in our linear operator characterization.

\begin{definition} \label{ev_def}
    Fix $\lambda \in \N_0^m$ and $(a,b) = (a_1,b_1,\ldots,a_m,b_m) \in \C^{2m}$ such that $(a_j,b_j) \neq (0,0)$ for all $j \in [m]$. Let $\ev_{(a,b)}: V(\lambda) \rightarrow V(0) \cong \C$ be the evaluation operator which maps $p$ to $p(a,b) = p(a_1,b_1,\ldots,a_m,b_m)$. We call $\Symb(\ev_{(a,b)}) \in V(\lambda)$ the \emph{evaluation symbol with root $(a,b)$}. Further:
    \[
        \Symb(\ev_{(a,b)}) = \prod_{j=1}^m (b_jx_j - a_jy_j)^{\lambda_j} =: (bx - ay)^\lambda
    \]
\end{definition}

The main significance of this notion comes from the following result, which is essentially just a restatement of the symbol lemma (Lemma \ref{symb_lemma}) for evaluation symbols.

\begin{lemma}[Evaluation Symbol Lemma]\label{ev_lemma}
    Fix $\lambda \in \N_0^m$, $p \in V(\lambda)$, and $(a,b) = (a_1,b_1,\ldots,a_m,b_m) \in \C^{2m}$ such that $(a_j,b_j) \neq (0,0)$ for all $j \in [m]$. Considering $(bx-ay)^\lambda$, the evaluation symbol with root $(a,b)$, we have:
    \[
        D^\lambda((bx-ay)^\lambda \otimes p) = (\lambda!)^2 p(a,b) \otimes 1 = (\lambda!)^2 p(a,b)
    \]
\end{lemma}

% \begin{remark}
%     It should be noted here the abuse of notation with elements $(a:b)$ of $\CP^1$. By definition, $(a:b) = (ca:cb)$ as elements of $\CP^1$ for any $0 \neq c \in \C$. However, these two different expressions of the same element of $\CP^1$ would yield two different evaluation symbols. Note though, that these evaluation symbols would be equal up to scalar. Since Grace's theorem is about zero vs. nonzero values, we will often be unconcerned with scalar differences, and this is why we allow the abuse of notation.
% \end{remark}

In what follows, we will extend Grace's theorem in a number of ways, mainly relying on the previous lemma and the symbol lemma itself. As we will see, the representation theoretic mentality combined with repeated use of the symbol lemma will yield many of the results of this paper with surprising simplicity.

We now obtain an interesting corollary of Grace's theorem, making use of the notion of a disjoint Grace pair. This particular formulation of the theorem will serve as a model for our linear operator characterization in \S\ref{str_Cstab_subsect}.

\begin{corollary}\label{ev_graces}
    Fix $\lambda \in \N_0^m$, $q \in V(\lambda)$, and any disjoint Grace pair $(C_1 \times \cdots \times C_m, B_1 \times \cdots \times B_m)$. Then the following are equivalent.
    \begin{enumerate}[label=(\roman*)]
        \item $D^\lambda(p \otimes q) \neq 0$ for all $(C_1 \times \cdots \times C_m)$-stable $p \in V(\lambda)$.
        \item $D^\lambda(p \otimes q) \neq 0$ for all $(C_1 \times \cdots \times C_m)$-stable evaluation symbols $p \in V(\lambda)$.
        \item $q$ is $(B_1 \times \cdots \times B_m)$-stable.
    \end{enumerate}
\end{corollary}
\begin{proof}
    $(i) \Rightarrow (ii)$ Immediate.
    
    $(ii) \Rightarrow (iii)$ Fix any $(a,b) = (a_1,b_1,\ldots,a_m,b_m) \in \C^{2m}$ such that $(a_k:b_k) \in B_k$ for all $k \in [m]$. Since $B_k$ and $C_k$ are disjoint, we have that $(a_k:b_k) \not\in C_k$ for all $k$, and therefore $\Symb(\ev_{(a,b)}) = (bx-ay)^\lambda$ is an evaluation symbol which is $(C_1 \times \cdots \times C_m)$-stable. The evaluation symbol lemma given above then implies:
    \[
        0 \neq D^\lambda(\Symb(\ev_{(a:b)}) \otimes q) = (\lambda!)^2 q(a,b) \otimes 1 = (\lambda!)^2 q(a,b)
    \]
    That is, $q(a,b) \neq 0$ for any $(a,b) \in \C^{2m}$ such that $(a:b) \in B_1 \times \cdots \times B_m$, and this implies $q$ is $(B_1 \times \cdots \times B_m)$-stable.
    
    $(iii) \Rightarrow (i)$ This follows immediately from the definition of Grace pair (Definition \ref{grace_pair_def}).
\end{proof}

\section{Stability Properties of Complex Linear Operators}\label{Cops_sect}

In \cite{bb1}, Borcea and Br{\"a}nd{\'e}n were concerned with classifying the class of weak $\Omega$-stability preserving operators, where $\Omega$ is some product of open circular regions. What they found is that an operator preserves weak $\Omega$-stability if a particular associated polynomial (what they called the symbol) is $\Omega$-stable. However, the ``only if'' direction does not necessarily hold. In particular, there are some weak $\Omega$-stability preserving operators for which the corresponding symbol is not $\Omega$-stable. They then showed that this could only happen under very specific circumstances: the operator must have image of dimension at most one.

Here, we will characterize all \emph{strong} $\Omega$-stability preserving linear operators (for a bit more general $\Omega$), as well as linear operators which map between different stability regions. And, as it turns out, the extra premise of strong stability preservation is exactly what is needed to have symbol stability be an equivalent condition. In a way, this makes sense: weak $\Omega$-stability preservation counts the zero polynomial as $\Omega$-stable, which in turn corresponds to potential zeros of the symbol in the region of stability. This does not happen with strong stability preservation, allowing for a more straightforward characterization.

First though, let's take a closer look at the Borcea-Br{\"a}nd{\'e}n characterization of weak stability-preserving linear operators.

\subsection{Weak Stability Preservation}

Borcea and Br{\"a}nd{\'e}n define the following symbol:
\[
    \Symb_{BB}(T) := T[(x+z)^\lambda] = \sum_{\mu \leq \lambda} \binom{\lambda}{\mu} z^{\lambda-\mu} T(x^\mu)
\]
They then obtain the following characterization of stability-preserving linear operators.

\newtheorem*{BB_Cops_thm}{Theorem \ref{BB_Cops_thm}}
\begin{BB_Cops_thm}[Borcea-Br{\"a}nd{\'e}n]
    Fix $\lambda \in \N_0^m$ and any linear operator $T: \C^\lambda[x_1,\ldots,x_m] \to \C[x_1,\ldots,x_m]$. The following are equivalent.
    \begin{enumerate}[label=(\roman*)]
        \item $T$ maps $\UHP^m$-stable polynomials to weakly $\UHP^m$-stable polynomials.
        \item One of the following holds:
            \begin{enumerate}[label=(\alph*)]
                \item $\Symb_{BB}(T)$ is $\UHP^{2m}$-stable.
                \item $T$ has image of dimension at most one, and is of the form
                \[
                    T: p \mapsto q \cdot \psi(p)
                \]
                where $q \in \C[x_1,\ldots,x_m]$ is $\UHP^m$-stable, and $\psi$ is some linear functional.
            \end{enumerate}
    \end{enumerate}
\end{BB_Cops_thm}

Using our terminology, this is a characterization of weak stability-preserving linear operators. This fantastic result perhaps has but one unfortunate piece: the degeneracy condition $(ii)(b)$. Its necessity is demonstrated in the following.

\begin{example}
    Define $T: \C^n[x] \to \C[x]$ via:
    \[
        T: \sum_{k=0}^n \binom{n}{k} a_k x^k \mapsto (a_n + a_{n-2}) x^n
    \]
    This operator obviously preserves weak $\UHP$-stability. We then have that $\Symb_{BB}(T) = (z^2 + 1)x^n$, which is not $\UHP^2$-stable.
\end{example}

As we will see below, this condition can be removed once we only consider strong stability-preserving operators. So then, maybe strong stability is the more natural notion? However ``natural'' it is, unfortunately it leaves out operators one might wish to consider. The most fundamental of such operators is the derivative operator $\partial_x$. While $\partial_x$ preserves strong $\overline{\UHP}$-stability, it only preserves weak $\UHP$-stability. Specifically, $1 \in \C^n[x]$ is $\UHP$-stable (all its roots are at $\infty$), but $\partial_x 1 \equiv 0$. With this, one obviously wants to be able to include weak stability preserving operators in any characterization of $\UHP$-stability preserving operators. We discuss how to use our strong stability preservation characterization to deal with operators like $\partial_x$ in Example~\ref{deriv_ex}.

\subsection{Strong Stability Preservation}\label{str_Cstab_subsect}

We now state one of our main characterization results, the strong stability preservation characterization. We then derive the Borcea-Br{\"a}nd{\'e}n characterization as a corollary.

\begin{theorem}\label{Cops_graces}
    Fix $\lambda \in \N_0^m$, $\alpha \in \N_0^l$, a linear operator $T \in \Hom(V(\lambda),V(\alpha))$, any disjoint Grace pair $(C_1 \times \cdots \times C_m, B_1 \times \cdots \times B_m)$, and any sets $S_1,\ldots,S_l \subseteq \CP^1$. The following are equivalent.
    \begin{enumerate}[label=(\roman*)]
        \item $T$ maps $(C_1 \times \cdots \times C_m)$-stable polynomials to nonzero $(S_1 \times \cdots \times S_l)$-stable polynomials.
        \item $T$ maps $(C_1 \times \cdots \times C_m)$-stable evaluation symbols to nonzero $(S_1 \times \cdots \times S_l)$-stable polynomials.
        \item $\Symb(T)$ is $(B_1 \times \cdots \times B_m) \times (S_1 \times \cdots \times S_l)$-stable.
    \end{enumerate}
\end{theorem}

One should notice the generality of this result in terms of stability regions. First note that any disjoint Grace pair can be considered, without altering the symbol in any way (e.g., via conjugation by M{\"o}bius transformations). And further, the output sets that can be considered have no restrictions whatsoever. The power of these extra features can be seen in the following examples, which demonstrate classical results regarding polynomial convolutions in a very symbol-oriented way.

\begin{example}\label{add_conv_ex1}
    Fix $p,q \in V(n)$, so that $(z_j:1)$ are the roots in $\CP^1$ of $q$ for $j \in [n]$. So, $q$ has no roots at $\infty$. The additive (Walsh) convolution of $p$ and $q$ is defined via:
    \[
        p *_+^n q := \frac{1}{n!} \sum_{k=0}^n \partial_x^k p \cdot (\partial_x^{n-k} q)(0,1)
    \]
    With this, $T_q(p) := p *_+^n q$ is a linear operator in $\Hom(V(n),V(n))$, and we have:
    \[
        \Symb(T_q) = \prod_{j=1}^n (xw - (z + z_jw)y) = \Hmg_{(n,n)}\left[\prod_{j=1}^n (x - (z + z_j))\right]
    \]
    Let $C \subset \CP^1$ be any projectively convex region, and define $S := \bigcup_j (C + z_j)$. If we order the input variables of $\Symb(T_q)$ as $(z,w,x,y)$, it is then straightforward to show that $\Symb(T_q)$ is $C \times (\CP^1 \setminus S)$-stable. (First deal with possible $(x:y) = (1:0)$ or $(z:w) = (1:0)$ cases, and then assume $y = w = 1$ to simplify the remaining cases.) Applying the previous theorem, this implies $T_q$ maps polynomials with roots in $C$ to polynomials with roots in $S$. (This is Theorem 5.3.1 in \cite{rahman2002anthpoly}.) Picking $C = \overline{\LHP}$ and real-rooted $q$ implies $T_q$ maps $\UHP$-stable polynomials to $\UHP$-stable polynomials. Restricting to $p \in V_\R(n)$ then shows that $T_q$ preserves real-rootedness.
\end{example}

\begin{example}\label{mult_conv_ex1}
    Fix $p,q \in V(n)$, so that $(z_j:1) \neq 0$ are the roots of $q$ for $j \in [n]$. So, $q$ has no roots at $0$ or $\infty$. The multiplicative (Grace-Szeg{\H{o}}) convolution of $p$ and $q$ (with coefficients $p_k$ and $q_k$, respectively) is defined via:
    \[
        p *_\times^n q := \sum_{k=0}^n \binom{n}{k}^{-1} (-1)^k p_k q_k x^k y^{n-k}
    \]
    With this, $T_q(p) := p *_\times^n q$ is a linear operator in $\Hom(V(n),V(n))$, and we have:
    \[
        \Symb(T_q) = \prod_{j=1}^n (xw - z_jzy) = \Hmg_{(n,n)}\left[\prod_{j=1}^n (x - z_j z)\right]
    \]
    Let $C \subset \CP^1$ be any projectively convex region, and define $S := \bigcup_j (z_j \cdot C)$. If we order the input variables of $\Symb(T_q)$ as $(z,w,x,y)$, it is then straightforward to show that $\Symb(T_q)$ is $C \times (\CP^1 \setminus S)$-stable. (As above, first deal with possible $(x:y) = (1:0)$ or $(z:w) = (1:0)$ cases, and then assume $y = w = 1$ to simplify the remaining cases.) Applying the previous theorem, this implies $T_q$ maps polynomials with roots in $C$ to polynomials with roots in $S$. (This is Theorem 3.4.1d in \cite{rahman2002anthpoly}.) Picking $C = \LHP \cup \R_+$ and $q$ with only positive roots implies $T_q$ maps $(\UHP \cup \overline{\R_-})$-stable polynomials to $(\UHP \cup \overline{\R_-})$-stable polynomials. Restricting to $p \in V_\R(n)$ then shows that $T_q$ preserves positive-rootedness.
\end{example}

In order to prove the above theorem, we need an operator-theoretic corollary to Grace's theorem. The following result is the main motivation for the symbol lemma (Lemma \ref{symb_lemma}), and demonstrates just how closely Grace's theorem relates to stability properties of linear operators. Further, it gives a slightly stronger result in one direction of the above characterization, as Grace pair disjointness is not a required premise.

\begin{proposition}\label{Cops_prop}
    Fix $\lambda \in \N_0^m$, $\alpha \in \N_0^l$, a linear operator $T \in \Hom(V(\lambda),V(\alpha))$, any Grace pair $(C_1 \times \cdots \times C_m, B_1 \times \cdots \times B_m)$, and any sets $S_1,\ldots,S_l \subseteq \CP^1$. If $\Symb(T)$ is $(B_1 \times \cdots \times B_m) \times (S_1 \times \cdots \times S_l)$-stable, then $T$ maps $(C_1 \times \cdots \times C_m)$-stable polynomials to nonzero $(S_1 \times \cdots \times S_l)$-stable polynomials.
\end{proposition}
\begin{proof}
    Fix any $(C_1 \times \cdots \times C_m)$-stable $q \in V(\lambda)$ and any $(c,d) = (c_1,d_1,\ldots,c_l,d_l) \in \C^{2l}$ such that $(c_j:d_j) \in S_j$ for all $j \in [l]$. Let $k_{\lambda,\alpha} := (-1)^\alpha(\lambda!)^2(\alpha!)^2$. The evaluation symbol lemma (Lemma \ref{ev_lemma}) and the symbol lemma (Lemma \ref{symb_lemma}) then give us the following expression of $T(q)$ evaluated at $(c,d) = (c_1,d_1,\ldots,c_l,d_l)$:
    \[
        \begin{split}
            k_{\lambda,\alpha} T(q)(c,d) &= k_{\lambda,0} D^\alpha\big(T(q) \otimes (dx-cy)^\alpha\big) \\
                &= D^{\lambda \sqcup \alpha}\big(\Symb(T) \otimes q \cdot (dx-cy)^\alpha\big) \\
                &= k_{0,\alpha} D^\lambda\big(\Symb(T)(z,w,c,d) \otimes q(z,w)\big)
        \end{split}
    \]
    In the last expression above, $D^\lambda$ acts on the variables $(z,w) = (z_1,w_1,\ldots,z_m,w_m)$. Since $r(z,w) := \Symb(T)(z,w,c,d)$ is $(B_1 \times \cdots \times B_m)$-stable and $q(z,w)$ is $(C_1 \times \cdots \times C_m)$-stable, we have that the last expression above is nonzero by definition of Grace pair (Definition \ref{grace_pair_def}). This implies $T(q)$ is $(S_1 \times \cdots \times S_l)$-stable.
\end{proof}

With this, we now give the proof of Theorem \ref{Cops_graces}.

\begin{proof}[Proof of Theorem \ref{Cops_graces}]
    The statement of this result, as well as its proof, is quite similar to that of the evaluation symbol version of Grace's theorem given in Corollary \ref{ev_graces}. We explicitly give the proof anyway, as it is rather short and straightforward.
    
    $(i) \Rightarrow (ii)$. Immediate.
    
    $(ii) \Rightarrow (iii)$. Fix $(a,b) = (a_1,b_1,\ldots,a_m,b_m) \in \C^{2m}$ such that $(a_j:b_j) \in B_j$ for all $j \in [m]$, and fix $(c,d) = (c_1,d_1,\ldots,c_l,d_l) \in \C^{2l}$ such that $(c_j:d_j) \in S_j$ for all $j \in [l]$. Let $k_{\lambda,\alpha} := (-1)^\alpha(\lambda!)^2(\alpha!)^2$. Since $B_j$ and $C_j$ are disjoint, we have that $(a_j:b_j) \not\in C_j$ for all $j$, and therefore $\Symb(\ev_{(a,b)}) = (bx-ay)^\lambda \in V(\lambda)$ is an evaluation symbol which is $(C_1 \times \cdots \times C_m)$-stable. Using the evaluation symbol lemma (Lemma \ref{ev_lemma}) and the symbol lemma (Lemma \ref{symb_lemma}), we compute:
    \[
        (-1)^\lambda k_{\lambda,\alpha} \Symb(T)(a,b,c,d) = D^{\lambda \sqcup \alpha}(\Symb(T) \otimes (bx-ay)^\lambda (dx-cy)^\alpha) = k_{\lambda,\alpha} T[(bx-ay)^\lambda](c,d)
    \]
    By $(ii)$ the last expression is nonzero, and thus $\Symb(T)(a,b,c,d) \neq 0$. This implies $\Symb(T)$ is $(B_1 \times \cdots \times B_m) \times (S_1 \times \cdots \times S_l)$-stable.
    
    $(iii) \Rightarrow (i)$. Proposition \ref{Cops_prop} above.
\end{proof}

As mentioned above, the previous proposition gives a slightly stronger result in the (symbol stability $\Rightarrow$ operator stability) direction. Using it, we revisit the additive and multiplicative convolutions with a more algebraic/symbolic mentality.

\begin{example}
    By Definition \ref{symb_def}, the $\Symb$ map gives a bijection between certain spaces of linear operators and polynomials. So, we can uniquely define a linear operator by giving its symbol. Using this idea, we specify $T \in \Hom(V(n,n), V(n))$ by defining its symbol in $V(n,n,n)$ with variables $(z,w),(t,s),(x,y)$ as follows:
    \[
        \Symb(T) := \Hmg_{(n,n,n)}\left[(x - (z+t))^n\right] = (xws - (zs + tw)y)^n
    \]
    Now, let us consider the additive convolution $*_+^n$ as an element of $\Hom(V(n,n),V(n))$ in the following way. Since $V(n,n) \cong V(n) \boxtimes V(n)$, we define $*_+^n$ on elements $p \boxtimes q \in V(n) \boxtimes V(n)$ via $*_+^n(p \boxtimes q) := p *_+^n q$ and extend linearly. We then compute $\Symb(*_+^n)$ as follows:
    \[
        \Symb(*_+^n) = *_+^n\left[(zy-xw)^n \boxtimes (ty-xs)^n\right] = (xws - (zs + tw)y)^n
    \]
    That is, $*_+^n$ is the operator that has our desired symbol. Fixing any $a,b,c,d \in \R$ such that $a<b$ and $c<d$, we define the sets $C_1 := \overline{\UHP} \setminus (a,b)$, $C_2 := \overline{\UHP} \setminus (c,d)$, $B_1 := \LHP \cup [a,b]$, $B_2 := \LHP \cup [c,d])$, and $S := \overline{\UHP} \setminus [a+c,b+d]$. Proposition \ref{Cops_prop} then implies $p *_+^n q$ has all its roots in $[a+c,b+d]$ whenever $p,q \in V_\R(n)$ have all their roots in $(a,b)$ and $(c,d)$, respectively. For real-rooted $p,q$ of degree $n$, this implies:
    \[
        \minroot(p) + \minroot(q) \leq \minroot(p *_+^n q) \leq \maxroot(p *_+^n q) \leq \maxroot(p) + \maxroot(q)
    \]
    Notice that we actually get a bit more. For $(C_1 \times C_2)$-stable $r := \sum_j p_j \boxtimes q_j \in V(n) \boxtimes V(n) \cong V(n,n)$, we have that $*_+^n[r]$ is $S$-stable. That is, $*_+^n$ has stability properties as an operator in $\Hom(V(n,n),V(n))$, not just as a convolution between two polynomials in $V(n)$.
\end{example}

\begin{example}
    As in the previous example, we can consider the multiplicative convolution $*_\times^n$ as an element of $\Hom(V(n,n),V(n))$ by defining $*_\times^n(p \boxtimes q) := p *_\times^n q$ on elements $p \boxtimes q \in V(n) \boxtimes V(n) \cong V(n,n)$ and extending linearly. We then compute its symbol in $V(n,n,n)$ with variables $(z,w),(t,s),(x,y)$ as follows:
    \[
        \Symb(*_\times^n) = *_\times^n\left[(zy-xw)^n \boxtimes (ty-xs)^n\right] = (xws - zty)^n = \Hmg_{(n,n,n)}\left[(x-zt)^n\right]
    \]
    Fixing any $a,b,c,d \in \R_+$ such that $0<a<b$ and $0<c<d$, we define the sets $C_1$, $C_2$, $B_1$, and $B_2$ as in the previous example. We then define $S := \overline{\R} \setminus [ac, bd]$. Proposition \ref{Cops_prop} then implies $p *_\times^n q$ has all its real roots in $[ac,bd]$ whenever $p,q \in V_\R(n)$ have all their roots in $(a,b)$ and $(c,d)$, respectively. (Notice that we could not apply the proposition if $\UHP \subset S$ or $\LHP \subset S$.) Since Example \ref{mult_conv_ex1} implies $p *_\times^n q$ is positive-rooted (and hence, real-rooted) whenever $p$ and $q$ are, this implies:
    \[
        \minroot(p) \cdot \minroot(q) \leq \minroot(p *_\times^n q) \leq \maxroot(p *_\times^n q) \leq \maxroot(p) \cdot \maxroot(q)
    \]
    As in the previous example, we also obtain stability properties for $*_\times^n$ as an operator in $\Hom(V(n,n),V(n))$, and not just as a polynomial convolution.
\end{example}

Using similar techniques, we can also circumvent the issue that arises from the fact that $\partial_x$ only preserves weak stability.

\begin{example}\label{deriv_ex}
    For fixed $n \geq 1$, consider the operator $\partial_x \in \Hom(V(n),V(n-1))$. We compute:
    \[
        \Symb(\partial_x) = \partial_x[(zy-xw)^n] = -nw(zy-xw)^{n-1} = \Hmg_{(n,n-1)}\left[-n(z-x)^{n-1}\right]
    \]
    For any $a,b \in \R$ such that $a<b$, it is straightforward to see that $\Symb(\partial_x)$ is $(C \times B)$-stable for $C := \LHP \cup (a,b)$ and $B := \overline{\UHP} \setminus (a,b)$, where the variables are ordered $(z,w),(x,y)$. (Notice that this does not hold when $\infty \in C$, due to the $w$ factor in the symbol.) Since $(C,B)$ is a disjoint Grace pair, the Theorem \ref{Cops_graces} implies $\partial_x$ preserves strong $B$-stability.
    
    With this, let $f \in \C^n[x]$ be a $\UHP$-stable polynomial of degree $1 \leq m \leq n$, and let $p \in V(m)$ be its degree-$m$ homogenization. Then $p$ has no roots at infinity, and therefore there exists $a < b$ such that $p$ is $\big(\overline{\UHP} \setminus (a,b)\big)$-stable. The previous discussion implies $\partial_x p$ is $\big(\overline{\UHP} \setminus (a,b)\big)$-stable, and in particular $\partial_x p$ is $\UHP$-stable. Since $\partial_x$ commutes with homogenization, this also implies $\partial_x f$ is $\UHP$-stable.
\end{example}

Other issues related to weak stability preservation can be dealt with in a similar way, by considering stability regions with small intervals in $\overline{\R}$ about $\infty$ attached. More generally though, the Borcea-Br{\"a}nd{\'e}n characterization ends up being a corollary of Theorem \ref{Cops_graces}, which we discuss and demonstrate now.

\subsection{Deriving the Complex Borcea-Br{\"a}nd{\'e}n Characterization}\label{Cops_deriv_subsect}

As mentioned above, we hope to obtain the Borcea-Br{\"a}nd{\'e}n characterization from our strong stability characterization given in Theorem \ref{Cops_graces}. To this end, we state two corollaries to Theorem \ref{Cops_graces}, which look (naively) as close to the Borcea-Br{\"a}nd{\'e}n characterization as possible. Let $C^c$ denote the complement of $C$ in $\CP^1$.

\begin{corollary}\label{naive_CBB_cor1}
    Fix $\lambda,\alpha \in \N_0^m$, a linear operator $T \in \Hom(V(\lambda),V(\alpha))$, and a Grace pair of the form $(C_1 \times \cdots \times C_m, C_1^c \times \cdots \times C_m^c)$. The following are equivalent.
    \begin{enumerate}[label=(\roman*)]
        \item $T$ preserves strong $(C_1 \times \cdots \times C_m)$-stability.
        \item $\Symb(T)$ is $(C_1^c \times \cdots \times C_m^c) \times (C_1 \times \cdots \times C_m)$-stable.
    \end{enumerate}
\end{corollary}

\begin{corollary}\label{naive_CBB_cor2}
    Fix $\lambda,\alpha \in \N_0^m$ and a linear operator $T \in \Hom(V(\lambda),V(\alpha))$. $T$ preserves strong stability iff $\Symb(T)$ is $(\overline{\LHP}^m \times \UHP^m)$-stable.
\end{corollary}

In Theorem \ref{BB_Cops_thm}, the analogous ``if'' direction of the previous corollary is paraphrased as follows: \emph{$T$ preserves weak stability if the Borcea-Br{\"a}nd{\'e}n symbol of $T$ is stable}. To see how this statement relates, we restate the definition of the Borcea-Br{\"a}nd{\'e}n symbol:
\[
    \Symb_{BB}(T) := T\left[(z+x)^\lambda\right] = \sum_{\mu \leq \lambda} \binom{\lambda}{\mu} z^{\lambda-\mu} T(x^\mu)
\]
Notice that by applying $z \mapsto -z$ and homogenizing, we obtain (up to scalar) the universal symbol $\Symb(T)$ defined in this paper. The crucial difference then is the fact that the Borcea-Br{\"a}nd{\'e}n ``if'' direction deals only with \emph{open} upper half-planes, whereas the previous corollary requires \emph{closed} half-plane stability of $\Symb(T)$ in the first $m$ pairs of variables. That is, the required premises of the ``if'' direction of the previous corollary are strictly stronger than that of the Borcea-Br{\"a}nd{\'e}n result.

These two results can be reconciled, however, which we now demonstrate. The following result provides the main link to the Borcea-Br{\"a}nd{\'e}n characterization, and it can be intuitively described as follows: with the exception of having a one-dimensional range, a linear operator which maps $(C_1 \times \cdots \times C_m)$-stable polynomials to weak $(B_1 \times \cdots \times B_m)$-stable polynomials can only have zeros on the boundary of the set of $(C_1 \times \cdots \times C_m)$-stable polynomials.

\begin{lemma}\label{bb_Clink}
    Fix $\lambda,\alpha \in \N_0^m$, a linear operator $T \in \Hom(V(\lambda),V(\alpha))$, and any open circular regions $C_1,\ldots,C_m,B_1,\ldots,B_m \subseteq \CP^1$. The following are equivalent.
    \begin{enumerate}[label=(\roman*)]
        \item $T$ maps $(C_1 \times \cdots \times C_m)$-stable polynomials to weakly $(B_1 \times \cdots \times B_m)$-stable polynomials.
        \item One of the following holds:
            \begin{enumerate}[label=(\alph*)]
                \item $T$ maps $(\overline{C_1} \times \cdots \times \overline{C_m})$-stable polynomials to nonzero $(B_1 \times \cdots \times B_m)$-stable polynomials.
                \item $T \equiv p_0 \cdot \psi$ for some weakly $(B_1 \times \cdots \times B_m)$-stable polynomial $p_0 \in V(\alpha)$ and some linear functional $\psi$.
            \end{enumerate}
    \end{enumerate}
\end{lemma}
\begin{proof}
    By appropriate $\SL_2(\C)$ action, we can assume WLOG that $C_k = B_k = \D$, the unit disc, for all $k \in [m]$.
    
    $(i) \Rightarrow (ii)$. We show that if $(a)$ is not the case, then $(b)$ must hold. It follows from $(i)$ that $T$ maps $\overline{\D}^m$-stable polynomials to (possibly identically zero) $\D^m$-stable polynomials. So, if $(a)$ is not the case, we have that $T(p) \equiv 0$ for some $\overline{\D}^m$-stable polynomial $p \in V(\lambda)$.
    
    The rest of the argument is essentially the proof of necessity found in \cite{bb1} for Theorem \ref{BB_Cops_thm}. Since the set of nonzero $\overline{\D}^m$-stable polynomials is open in $V(\lambda)$, there is some ball $B(p) \subset V(\lambda)$ centered at $p$ such that $B(p)$ contains only $\overline{\D}^m$-stable polynomials. So, $T[B(p)]$ is an open set in the image of $T$ containing 0 and otherwise consisting of $\D^m$-stable polynomials. Therefore, the image of $T$ is a vector space consisting of $\D^m$-stable polynomials. Lemma \ref{dim_lemma} (and appropriate $\SL_2(\C)$ action) then implies the image of $T$ is of dimension $\leq 1$, and $(b)$ follows.
    
    $(ii) \Rightarrow (i)$. If $(b)$ holds, then $(i)$ is immediate. Otherwise, fix $p \in V(\lambda)$ such that $p$ is $\D^m$-stable. For all $n \in \N$, define:
    \[
        p_n := p((1-n^{-1})x_1, y_1, (1-n^{-1})x_2, y_2, \ldots, (1-n^{-1})x_m, y_m)
    \]
    So, $p_n$ is $\overline{\D}^m$-stable for all $n$, and $\lim_{n \rightarrow \infty} p_n = p$ coefficient-wise. By $(ii)$, $T(p_n)$ is $\D^m$-stable for all $n$, and by continuity, $T(p) = \lim_{n \rightarrow \infty} T(p_n)$. Hurwitz's theorem then implies $T(p)$ is either identically zero or $\D^m$-stable.
\end{proof}

This lemma then yields the following corollaries to Theorem \ref{Cops_graces}. Applying the necessary maps to convert $\Symb$ to $\Symb_{BB}$ as discussed above, these results give precisely the Borcea-Br{\"a}nd{\'e}n characterization proven in Theorem \ref{BB_Cops_thm} and more generally in Theorem 6.3 of \cite{bb1}. In particular, Corollary \ref{CBB_cor1} can be seen as a unification of the complex characterization results of \cite{bb1}.

\begin{corollary}\label{CBB_cor1}
    Fix $\lambda,\alpha \in \N_0^m$, a linear operator $T \in \Hom(V(\lambda),V(\alpha))$, and any open circular regions $C_1, \ldots, C_m$. The following are equivalent.
    \begin{enumerate}[label=(\roman*)]
        \item $T$ preserves weak $(C_1 \times \cdots \times C_m)$-stability.
        \item One of the following holds:
            \begin{enumerate}[label=(\alph*)]
                \item $\Symb(T)$ is $(\overline{C_1}^c \times \cdots \times \overline{C_m}^c) \times (C_1 \times \cdots \times C_m)$-stable.
                \item $T \equiv p_0 \cdot \psi$ for some weakly $(C_1 \times \cdots \times C_m)$-stable polynomial $p_0 \in V(\alpha)$ and some linear functional $\psi$.
            \end{enumerate}
    \end{enumerate}
\end{corollary}
\begin{proof}
    The result follows from the Lemma \ref{bb_Clink} and Theorem \ref{Cops_graces} applied to an operator $T$ which maps $(\overline{C_1} \times \cdots \times \overline{C_m})$-stable polynomials to nonzero $(C_1 \times \cdots \times C_m)$-stable polynomials.
\end{proof}

\begin{corollary}\label{CBB_cor2}
    Fix $\lambda,\alpha \in \N_0^m$ and a linear operator $T \in \Hom(V(\lambda),V(\alpha))$. $T$ preserves weak stability iff one of the following holds:
    \begin{enumerate}[label=(\alph*)]
        \item $\Symb(T)$ is $(\LHP^m \times \UHP^m)$-stable.
        \item $T \equiv p_0 \cdot \psi$ for some weakly stable polynomial $p_0 \in V(\alpha)$ and some linear functional $\psi$.
    \end{enumerate}
\end{corollary}

Notice that our naive guess at strong stability results which emulate the Borcea-Br{\"a}nd{\'e}n characterization (Corollaries \ref{naive_CBB_cor1} and \ref{naive_CBB_cor2}) was incorrect. We actually needed to consider \emph{closed} circular stability regions $\overline{C_k}$, so that their complements in $\CP^1$ would be open (i.e., to ensure Grace pair disjointness, which is required to apply Theorem \ref{Cops_graces}). We see this play out in condition $(ii)(a)$ of Corollary \ref{CBB_cor1}.

\section{Stability Properties of Real Linear Operators}\label{Rops_sect}

Borcea and Br{\"a}nd{\'e}n also classified the class of weak real stability preserving linear operators. As in the complex case, they showed that weak real stability preservation of a linear operator $T$ is almost equivalent to real stability of the associated symbol $\Symb_{BB}(T)$. We have to say ``almost equivalent'' here because there are certain weak real stability preserving operators for which the corresponding symbol is not real stable. As before, this implies a certain dimension restriction: such operators must have image of dimension at most two.

We will now characterize all strong real stability preserving linear operators. As above, strong real stability preservation will serve to eliminate the degeneracy condition of the Borcea-Br{\"a}nd{\'e}n characterization. In this section, we duplicate the outline of our previous discussion on complex operators, making use of arguments similar to those found in \cite{bb1} to fill in the gaps.

Further, we also obtain a characterization of a certain class of operators which preserve ray- and interval-rootedness. The question of a full characterization of such operators is as of yet still an open problem (see \cite{bba}). Here, we answer this question for operators which preserve both strong ray- or interval-rootedness as well as weak real-rootedness.

\subsection{Weak Real Stability Preservation}

Borcea and Br{\"a}nd{\'e}n obtain the following characterization of weak real stability preserving linear operators. Recall the notion of \emph{proper position} (denoted by $\ll$) given in Definition \ref{pp_def}.

\newtheorem*{BB_Rops_thm}{Theorem \ref{BB_Rops_thm}}
\begin{BB_Rops_thm}[Borcea-Br{\"a}nd{\'e}n]
    Fix $\lambda \in \N_0^m$ and any linear operator $T: \R_\lambda[x_1,\ldots,x_m] \to \R[x_1,\ldots,x_m]$. The following are equivalent.
    \begin{enumerate}[label=(\roman*)]
        \item $T$ maps real stable polynomials to weakly real stable polynomials.
        \item One of the following holds:
            \begin{enumerate}[label=(\alph*)]
                \item $\Symb_{BB}(T)$ is $\UHP^{2m}$-stable.
                \item $\Symb_{BB}(T)$ is $(\LHP^m \times \UHP^m)$-stable.
                \item $T$ has image of dimension at most two, and is of the form
                \[
                    T: p \mapsto q \cdot \psi_1(p) + r \cdot \psi_2(p)
                \]
                where $q,r \in \R[x_1,\ldots,x_m]$ are weakly real stable such that $q \ll r$, and $\psi_1,\psi_2$ are real linear functionals.
            \end{enumerate}
    \end{enumerate}
\end{BB_Rops_thm}

As in the case of complex operators, the degeneracy condition $(ii)(c)$ is the result of allowing weak real stability preserving operators. We now give an example which demonstrates its necessity.

\begin{example}
    Define $T: \R_n[x] \to \R[x]$ via:
    \[
        T: \sum_{k=0}^n \binom{n}{k} a_k x^k \mapsto a_nx^n + a_{n-2}x^{n-1} = (a_nx + a_{n-2})x^{n-1}
    \]
    This operator obviously preserves weak real stability. We then have that $\Symb_{BB}(T) = (z^2+x)x^{n-1}$, which is not $\UHP^2$-stable nor $(\LHP \times \UHP)$-stable.
\end{example}

Again, the degeneracy condition is required for the characterization but obscures the connection between an operator and its symbol. To remove it, we now turn to our characterization of strong real stability preserving operators.

\subsection{Strong Real Stability Preservation}

We state and prove our strong real stability preservation characterization here, and then derive the Borcea-Br{\"a}nd{\'e}n characterization as a corollary. The proof here takes a bit more work than in the complex case, and will rely on many of the real stability results discussed in \S\ref{stab_subsect}. This extra work is essentially taken from the proof of Theorem \ref{BB_Rops_thm} found in \cite{bb1}.

\begin{theorem}\label{Rops_graces}
    Fix $\lambda \in \N_0^m$, $\alpha \in \N_0^l$, and a linear operator $T \in \Hom(V(\lambda),V(\alpha))$ such that $T$ restricts to a real linear operator from $V_\R(\lambda)$ to $V_\R(\alpha)$. The following are equivalent.
    \begin{enumerate}[label=(\roman*)]
        \item $T$ preserves strong real stability.
        \item $\Symb(T)$ is either $(\overline{\LHP}^m \times \UHP^l)$-stable or $(\overline{\LHP}^m \times \LHP^l)$-stable.
    \end{enumerate}
\end{theorem}
\begin{proof}
    $(i) \Rightarrow (ii)$. Fixing $(z_0:w_0) \in \overline{\LHP}^m$, we have that $(w_0x - z_0y)^\lambda$ is $\UHP^m$-stable. If $(z_0:w_0) \in \overline{\R}^m$, then $T[(w_0x - z_0y)^\lambda]$ is nonzero and real stable by assumption. Combining the symbol lemma (Lemma \ref{symb_lemma}) and the evaluation symbol lemma (Lemma \ref{ev_lemma}), this implies $\Symb(T)(z_0,w_0,x,y) = (-1)^\lambda T[(w_0x - z_0y)^\lambda]$ is both $\UHP^l$-stable and $\LHP^l$-stable.
    
    On the other hand, suppose $(z_0:w_0) \not\in \overline{\R}^m$. By Lemma \ref{rs_lemma}, we have that $T[(w_0x - z_0y)^\lambda]$ is either $\UHP^l$-stable, $\LHP^l$-stable, or zero. Now suppose there are $(z_0:w_0), (z_0':w_0') \in \overline{\LHP}^m \setminus \overline{\R}^m$ such that $T[(w_0x - z_0y)^\lambda]$ is $\UHP^l$-stable and $T[(w_0'x - z_0'y)^\lambda]$ is $\LHP^l$-stable. By a homotopy argument, there exists $(z_0'':w_0'') \in \overline{\LHP}^m \setminus \overline{\R}^m$ such that $T[(w_0''x - z_0''y)^\lambda]$ is $(\UHP^l \cup \LHP^l)$-stable or zero. By Lemma \ref{scal_rs_lemma}, $T[c_0(w_0''x - z_0''y)^\lambda]$ is either real stable or zero for some complex scalar $c_0 \neq 0$.

    Let $c_0(w_0''x - z_0''y)^\lambda = q(x,y) + ir(x,y)$ for $q,r \in V_\R(\lambda)$, which are both real stable or zero by Hermite-Biehler (Proposition \ref{hermite-biehler}). Note further that $r \not\equiv 0$ since $(z_0'':w_0'') \not\in \overline{\R}^m$. However, since $T(q + ir) = T(q) + iT(r)$ is real stable or zero and $T$ restricts to real linear operator, it must be that $T(r) \equiv 0$. This contradicts the fact that $T$ strongly preserves real stability.
    
    So, $\Symb(T)(z_0,w_0,x,y) = (-1)^\lambda T[(w_0x - z_0y)^\lambda]$ is either $\UHP^l$-stable (in the $x,y$ variables) for all $(z_0:w_0) \in \overline{\LHP}^m \setminus \overline{\R}^m$, or $\LHP^l$-stable for all $(z_0:w_0) \in \overline{\LHP}^m \setminus \overline{\R}^m$. Combining this with the $(z_0:w_0) \in \overline{\R}^m$ case, we have that $\Symb(T)$ is either $(\overline{\LHP}^m \times \UHP^l)$-stable or $(\overline{\LHP}^m \times \LHP^l)$-stable.
    
    $(ii) \Rightarrow (i)$. By the complex stability characterization (Theorem \ref{Cops_graces}), $T$ maps $\UHP^m$-stable polynomials to either nonzero $\UHP^l$-stable polynomials or nonzero $\LHP^l$-stable polynomials. Since $T$ restricts to a real linear operator on $V_\R(\lambda)$, $T$ preserves strong real stability.
\end{proof}

As a final note, the ``homotopy argument'' used in the previous proof is not quite that of the proof found in \cite{bb1}, though it is similar. Here, one just needs to be a bit more careful about the precise homotopy with respect to points at infinity.

\subsection{Deriving the Real Borcea-Br{\"a}nd{\'e}n Characterization}\label{Rops_deriv_subsect}

As in the complex case, we now obtain the Borcea-Br{\"a}nd{\'e}n weak real stability characterization as a corollary to our strong real stability characterization given in Theorem \ref{Rops_graces}. To this end, we start by giving a sort of real stability version of Lemma \ref{bb_Clink}. The proof of this lemma is similar in spirit to that of the strong real stability characterization given above.

\begin{lemma}\label{bb_Rlink}
    Fix $\lambda,\alpha \in \N_0^m$ and a linear operator $T \in \Hom(V(\lambda),V(\alpha))$ such that $T$ restricts to a real linear operator from $V_\R(\lambda)$ to $V_\R(\alpha)$. The following are equivalent.
    \begin{enumerate}[label=(\roman*)]
        \item $T$ preserves weak real stability.
        \item One of the following holds:
            \begin{enumerate}[label=(\alph*)]
                \item $T$ maps $\overline{\UHP}^m$-stable polynomials to nonzero $\UHP^m$-stable polynomials.
                \item $T$ maps $\overline{\UHP}^m$-stable polynomials to nonzero $\LHP^m$-stable polynomials.
                \item $T$ has image of dimension at most two, and is of the form
                \[
                    T: p \mapsto q \cdot \psi_1(p) + r \cdot \psi_2(p)
                \]
                where $q,r \in V_\R(\alpha)$ are weakly real stable such that $q \ll r$, and $\psi_1,\psi_2$ are real linear functionals.
            \end{enumerate}
    \end{enumerate}
\end{lemma}
\begin{proof}
    $(i) \Rightarrow (ii)$. By the complex characterization (Theorem \ref{Cops_graces}), we only need to consider evaluation symbols when demonstrating $(a)$ or $(b)$. For any $(z_0:w_0) \in \LHP^m$, Lemma \ref{rs_lemma} then implies $T[(w_0x-z_0y)^\lambda]$ is either $\UHP^m$-stable, $\LHP^m$-stable, or identically zero. We now show that $(c)$ holds if $(a)$ and $(b)$ do not.
    
    If neither $(a)$ nor $(b)$ holds for evaluation symbols, then there exist $(z_0:w_0),(z_0':w_0') \in \LHP^m$ such that $T[(w_0x-z_0y)^\lambda]$ is $\UHP^m$-stable or zero and $T[(w_0'x-z_0'y)^\lambda]$ is $\LHP^m$-stable or zero. As in the proof of the strong real stability characterization (Theorem \ref{Rops_graces}), a homotopy argument implies there exists $(z_0'':w_0'') \in \LHP^m$ such that $T[(w_0''x-z_0''y)^\lambda]$ is $(\UHP^m \cup \LHP^m)$-stable or zero. Lemma \ref{scal_rs_lemma} then implies $T[c_0(w_0''x-z_0''y)^\lambda]$ is real stable or zero, for some complex scalar $c_0 \neq 0$.
    
    For the sake of simplicity, we denote $q_0(x,y) = q_0(x_1,y_1,\ldots,x_m,y_m) := c_0(w_0''x-z_0''y)^\lambda$. Since the set of $\overline{\UHP}^m$-stable polynomials is open in $V(\lambda)$, let $B(0)$ be some open ball centered at $0$ such that $q_0 + iB(0)$ consists of nonzero $\overline{\UHP}^m$-stable polynomials. So, for any $r_0 \in B(0)$, Lemma \ref{rs_lemma} implies $T(q_0) + iT(r_0)$ is either $\UHP^m$-stable, $\LHP^m$-stable, or zero. Hermite-Biehler then implies $T(r_0)$ is real stable or zero whenever $r_0 \in B(0) \cap V_\R(\lambda)$. Therefore, $T[V_\R(\lambda)]$ consists of real stable polynomials, and $(c)$ follows from Lemma \ref{dim_lemma}.
    
    $(ii) \Rightarrow (i)$. If $(c)$ holds, then $(i)$ follows from Hermite-Biehler. Otherwise, suppose WLOG that $(a)$ holds. We can then use an argument similar in spirit to that of Lemma \ref{bb_Clink} to show that $T$ maps $\UHP^m$-stable polynomials to weakly $\UHP^m$-stable polynomials. Since $T$ restricts to a real operator, this implies $(i)$.
\end{proof}

As in Lemma \ref{bb_Clink}, we use the previous lemma to link the characterizations of weak and strong stability preserving operators as follows. Applying the necessary maps to convert $\Symb(T)$ to $\Symb_{BB}(T)$ below gives essentially the characterization of weak real stability preserving operators given in Theorem \ref{BB_Rops_thm}.

\begin{corollary}\label{bb_Rops}
    Fix $\lambda,\alpha \in \N_0^m$ and a linear operator $T \in \Hom(V(\lambda),V(\alpha))$ such that $T$ restricts to a real linear operator from $V_\R(\lambda)$ to $V_\R(\alpha)$. The following are equivalent.
    \begin{enumerate}[label=(\roman*)]
        \item $T$ preserves weak real stability.
        \item One of the following holds:
            \begin{enumerate}[label=(\alph*)]
                \item $\Symb(T)$ is $(\LHP^m \times \UHP^m)$-stable.
                \item $\Symb(T)$ is $(\LHP^m \times \LHP^m)$-stable.
                \item $T$ has image of dimension at most two, and is of the form
                \[
                    T: p \mapsto q \cdot \psi_1(p) + r \cdot \psi_2(p)
                \]
                where $q,r \in V_\R(\alpha)$ are weakly real stable such that $q \ll r$, and $\psi_1,\psi_2$ are real linear functionals.
            \end{enumerate}
    \end{enumerate}
\end{corollary}
\begin{proof}
    Apply the complex characterization (Theorem \ref{Cops_graces}) to conditions $(ii)(a)$ and $(ii)(b)$ of Lemma \ref{bb_Rlink} above.
\end{proof}

\subsection{Ray and Interval Stability}\label{Jops_subsect}

We now apply the above results to projectively convex regions of the form $\UHP \cup J^c$, where $J \subset \overline{\R}$ is some connected set. From this, we obtain a classification of operators which both preserve strong $J$-rootedness and weak real-rootedness (a polynomial $p \in V(n)$ is $J$-rooted if all its roots lie in $J$). This of course does not completely solve the open problem of providing a classification of interval- and ray-stability preserving operators (see, e.g., \cite{bba}). However, it does seem to be the natural corollary obtained by applying proof methods similar to that of \cite{bb1}.

That said, we now proceed to prove the main result of this subsection, Theorem \ref{Jops_graces}. We first start with a short-hand definition in order to simplify the proof.

\begin{definition}
    Fix $\lambda,\alpha \in \N_0^m$ and a linear operator $T \in \Hom(V(\lambda),V(\alpha))$ such that $T$ restricts to a real linear operator and preserves weak real-stability. We say $T$ is \emph{degenerate} if it satisfies condition $(ii)(c)$ of Corollary \ref{bb_Rops}.
\end{definition}

We now prove two lemmas. The first is straightforward, but rather interesting in its own right.

\begin{lemma}\label{conv_lemma}
    Fix a closed bounded interval $J \subset \R$ and a subspace $W \subseteq V_\R(n)$ consisting of weakly real-rooted polynomials. Let $S \subseteq W$ denote the subset of top-degree monic $J$-rooted polynomials. There exist $p,q \in S$ such that $p \ll q$ and $S$ is the convex hull of $p$ and $q$.
\end{lemma}
\begin{proof}
    Lemma \ref{dim_lemma} implies $W$ is of dimension at most two, and so then Lemma \ref{total_order_lemma} implies the relation $\ll$ is a total order on $S$. Applying the root ordering property of Lemma \ref{total_order_lemma}, the closedness of $S$ implies there are $p,q \in S$ such that $p \ll q$ and $p \ll r \ll q$ for all $r \in S$. Basic sign arguments and the fact that $S$ is contained in the span of $\{p,q\}$ then imply $S$ is the the convex hull of $\{p,q\}$.
\end{proof}

The second lemma is perhaps less straightforward in terms of proof, but follows from the following intuitive idea: an open ball in some complex subspace of polynomials yields, roughly speaking, an open ball of zeros.

\begin{lemma}
    Fix $n,m \in \N_0$ and a linear operator $T \in \Hom(V(n),V(m))$ which restricts to a real linear operator and preserves weak real-rootedness. If there exist some $\overline{\UHP}$-stable $p_0 \in V(n)$ and some $(x_0:y_0) \in \overline{\R}$ such that $T(p_0)(x_0,y_0) = 0$, then one of the following holds:
    \begin{enumerate}[label=(\alph*)]
        \item $T(p_0)$ is real-rooted or identically zero.
        \item $T(p)(x_0,y_0) = 0$ for all $p \in V(n)$.
    \end{enumerate}
\end{lemma}
\begin{proof}
    Let $q_0,r_0 \in V_\R(m)$ be such that $T(p_0) = q_0 + ir_0$. Also suppose that $T(p_0) \not\equiv 0$ and that $(b)$ does not hold, and let $p_1$ be such that $T(p_1)(x_0,y_0) \neq 0$. WLOG, we may also assume $p_1 \in V_\R(n)$ by considering its real or imaginary part. We will now prove that $T(p_0)$ must be real-rooted.
    
    First, suppose further that $T(p_0)$ has a multiple root at $(x_0,y_0)$. For small fixed $\epsilon$, $p_0 + \epsilon p_1$ is $\overline{\UHP}$-stable and so Lemma \ref{rs_lemma} implies $T(p_0 + \epsilon p_1)$ is either $\UHP$-stable or $\LHP$-stable. Hermite-Biehler (Proposition \ref{hermite-biehler}) then implies $q_0 + \epsilon T(p_1)$ and $r_0$ have interlacing roots. However, since $T$ restricts to a real linear operator, it must be that $q_0$ and $r_0$ both have a multiple root at $(x_0,y_0)$. The fact that $q_0 + \epsilon T(p_1)$ has no root at $(x_0,y_0)$ yields a contradiction, as interlacing is then impossible.
    
    Otherwise, $T(p_0)$ has a simple root at $(x_0,y_0)$. Define $R \in \Hom(V(n),V(1))$ via $R := d_\phi^{n-1} \circ T$, where $\phi \in \SL_2(\C)$ is such that $(x_0:y_0)$ is the pole of $\phi$. We have that $R(p_0)(x_0,y_0) = 0$, but $R(p_0) \not\equiv 0$ since the root is simple. Further, $R(p_1)(x_0,y_0) \neq 0$, and therefore $R$ is a surjective continuous linear map. By the open mapping theorem, there exists a one-real-dimensional curve $\Gamma \subset V(n)$ through $p_0$, for which $R(\Gamma)$ contains elements with root in $\UHP$ on one side of $p_0$ (call this side $\Gamma_+$) and elements with root in $\LHP$ on the other side (call it $\Gamma_-$). So by Laguerre's theorem (Proposition \ref{laguerre_prop}), elements of $T(\Gamma_+)$ have some roots in $\UHP$ and elements of $T(\Gamma_-)$ have some roots in $\LHP$. Since polynomials near $p_0$ are $\overline{\UHP}$-stable, Lemma \ref{rs_lemma} implies elements of $T(\Gamma \cap B_\epsilon(p_0))$ are all $\UHP$-stable or $\LHP$-stable for some small ball $B_\epsilon(p_0)$ about $p_0$. So elements of $T(\Gamma_+ \cap B_\epsilon(p_0))$ are $\LHP$-stable and elements of $T(\Gamma_- \cap B_\epsilon(p_0))$ are $\UHP$-stable, and therefore $T(p_0)$ is real-rooted.
\end{proof}

We now prove our main result on ray- and interval-stability preserving operators. First we state the theorem for closed bounded output intervals, as it clarifies the proof quite a bit. We will then extend the result to other connected regions in $\overline{\R}$.

\begin{theorem}\label{Jops_graces}
    Fix $n,m \in \N_0$ and a linear operator $T \in \Hom(V(n),V(m))$ which restricts to a real linear operator. Further, let $I \subseteq \R$ be any interval, and let $J \subset \R$ be any closed bounded interval. The following are equivalent.
    \begin{enumerate}[label=(\roman*)]
        \item $T$ preserves weak real-rootedness and maps $I$-rooted polynomials to nonzero $J$-rooted polynomials.
        \item One of the following holds:
            \begin{enumerate}[label=(\alph*)]
                \item $\Symb(T)$ is $(\LHP \cup I) \times (\overline{\UHP} \setminus J)$-stable.
                \item $\Symb(T)$ is $(\LHP \cup I) \times (\overline{\LHP} \setminus J)$-stable.
                \item $T$ has image of dimension at most two, and is of the form
                \[
                    T: p \mapsto q \cdot \psi_1(p) + r \cdot \psi_2(p)
                \]
                where $q,r \in V_\R(m)$ are top-degree monic and weakly $J$-rooted such that $q \ll r$, and $\psi_1$ and $\psi_2$ are real linear functionals such that $\psi_1(p) \cdot \psi_2(p) \geq 0$ (not both zero) holds for any $I$-rooted $p$.
            \end{enumerate}
    \end{enumerate}
\end{theorem}
\begin{proof}
    $(i) \Rightarrow (ii)$. Suppose $T$ is nondegenerate. So, $\Symb(T)$ is either $(\LHP \times \UHP)$-stable or $(\LHP \times \LHP)$-stable by Corollary \ref{bb_Rops}. By Lemma \ref{bb_Rlink}, either $T$ maps $\overline{\UHP}$-stable evaluation symbols entirely to nonzero $\UHP$-stable polynomials or entirely to nonzero $\LHP$-stable polynomials. If for some $(z_0:w_0) \in \LHP$ we have that $T[(w_0x-z_0y)^n]$ has a root in $\overline{\R}$, then we can apply the previous lemma. If condition $(a)$ of the lemma holds, then $T[(w_0x-z_0y)^n]$ is real-rooted or identically zero. The proof of Lemma \ref{bb_Rlink} then implies $T$ is degenerate, a contradiction. Otherwise condition $(b)$ of the lemma holds, and therefore the real roots of $T[(w_0x-z_0y)^n]$ must be in $J$. So in fact, $T$ maps $\overline{\UHP}$-stable evaluation symbols entirely to nonzero $(\overline{\UHP} \setminus J)$-stable polynomials or entirely to nonzero $(\overline{\LHP} \setminus J)$-stable polynomials. Finally, $T$ maps $I$-rooted evaluation symbols to nonzero $(\overline{\UHP} \setminus J)$-stable and $(\overline{\LHP} \setminus J)$-stable polynomials by assumption. The complex characterization (Theorem \ref{Cops_graces}) then implies $(a)$ or $(b)$.
    
    Otherwise, $T$ is degenerate and $T[V_\R(n)]$ consists entirely of real-rooted polynomials. Condition $(c)$ follows from Lemma \ref{conv_lemma}.
    
    $(ii) \Rightarrow (i)$. By Corollary \ref{bb_Rops}, $T$ preserves weak real-rootedness. If $(a)$ or $(b)$ holds, then the complex characterization (Theorem \ref{Cops_graces}) and the fact that $T$ restricts to a real operator imply $T$ maps $I$-rooted polynomials to nonzero $J$-rooted polynomials.
    
    Otherwise $(c)$ holds. For any real-rooted $p$, let $\lambda(p)$ and $\mu(p)$ denote the largest and smallest roots of $p$, respectively. Since $q,r$ are top-degree monic, every convex combination of $q$ and $r$ has all its roots in the interval $[\mu(q), \lambda(r)] \subseteq J$. Since $\psi_1 \cdot \psi_2 \geq 0$ (not both zero) holds for $I$-rooted polynomials, we have that $T$ maps $I$-rooted polynomials to nonzero $J$-rooted polynomials.
\end{proof}

Notice that this result immediately holds for other closed, connected regions $I,J \subset \overline{\R}$ by the action of some appropriate $\phi \in \SL_2(\R)$. In fact, one can directly apply the action of $\phi$ to conditions $(ii)(a)$ and $(ii)(b)$, due to the fact that our definition of the ``universal'' symbol works for any projectively convex regions. The only significant change comes when applying $\phi$ to condition $(ii)(c)$. Further, the only issue with $(ii)(c)$ as it is written now is the requirement that $p_1$ and $p_2$ be top-degree monic polynomials. Having zeros at infinity, for instance, means that a polynomial cannot ever be top-degree monic (as the leading homogeneous coefficient is 0). There are ways to rewrite $(ii)(c)$ that avoids this problem, but it is probably more intuitive to state the result as above and apply $\phi \in \SL_2(\R)$.

Additionally, the result holds for open and half-open bounded intervals $J \subset \R$, with a bit of tweaking to condition $(ii)(c)$. (Again, the universality of the symbol means that $(ii)(a)$ and $(ii)(b)$ remain unchanged.) We state this in the following, where the action of $\phi \in \SL_2(\R)$ can be used to obtain similar results regarding open and half-open connected regions in $\overline{\R}$.

\begin{corollary}\label{Jops_graces_cor}
    The previous theorem holds when $J \subset \R$ is an open (or half-open) bounded interval, given the following alterations to condition $(ii)(c)$: if the image of $T$ is of dimension exactly two, then $p_1,p_2 \in V_\R(m)$ are top-degree monic $\overline{J}$-rooted polynomials such that the largest root of $p_1$ and the smallest root of $p_2$ are in $J$ (for $p_1 \ll p_2$), and $\psi_1 \neq 0$ (resp. $\psi_2 \neq 0$) whenever $p_2$ (resp. $p_1$) is not $J$-rooted.
\end{corollary}
\begin{proof}
    The condition that the largest root of $p_1$ and the smallest root of $p_2$ are in $J$ (and the fact that $p_1 \ll p_2$) implies that $\alpha p_1 + \beta p_2$ is $J$-rooted for all $\alpha,\beta > 0$. Applying Lemma \ref{conv_lemma} to $\overline{J}$ completes the proof.
\end{proof}

We now give a few examples. The first demonstrates the necessity of the premise that $T$ preserves weak real-rootedness.

\begin{example}
    Consider the operator $T_n: V(n) \to V(n)$ defined via:
    \[
        T_n: x^ky^{n-k} \mapsto \Hmg_n[x(x-1)(x-2)\cdots(x-k+1)]
    \]
    By Proposition 7.31 in \cite{fisk2006polys}, $T_n$ preserves positive-rootedness for all $n$. However, $T_2$ does not preserve real-rootedness, for example. In particular:
    \[
        T_2(x^2+2xy+y^2) = x(x-y) + 2xy + y^2 = x^2 + xy + y^2
    \]
    We now compute the symbol of $T_2$:
    \[
        \Symb(T_2) = T_2[(xw-zy)^2] = x(x-y)w^2 - 2xzyw + z^2y^2 = \Hmg_{(2,2)}[(x-z)^2 - x]
    \]
    Notice that for $x = -1$, we have that $(-1-z)^2 + 1$ is not real rooted. Therefore, $\Symb(T_2)$ is neither $(\LHP \cup (0,\infty)) \times (\overline{\UHP} \setminus (0,\infty))$-stable nor $(\LHP \cup (0,\infty)) \times (\overline{\LHP} \setminus (0,\infty))$-stable when the variables are ordered $(z,w),(x,y)$. That is, the operator $T_2$ does not contradict the previous theorem.
\end{example}

In the second example, we demonstrate root preservation properties of $f(\partial_x)$ for real-rooted $f$. These are standard results of the classical theory: see, e.g., Corollary 5.4.1 in \cite{rahman2002anthpoly}.

\begin{example}
    For any real-rooted $f \in \C[x]$, consider the operator $D_f \in \Hom(V(n),V(n))$ defined via $D_f: g \mapsto f(y\partial_x)g$ (i.e., the homogenized version of $f(\partial_x)$). To determine properties of this operator, we first write:
    \[
        f(y\partial_x) = c_0 \prod_{j=1}^m (y\partial_x - \alpha_j)
    \]
    Here, the $\alpha_j \in \R$ are the roots of $f$. Next, we compute the symbol of $(y\partial_x - \alpha_j) \in \Hom(V(n),V(n))$ for $j \in [m]$:
    \[
        \begin{split}
            \Symb(y\partial_x - \alpha_j) &= (y\partial_x - \alpha_j)(zy-xw)^n \\
                &= -(\alpha_j (zy-xw) + nwy)(zy-xw)^{n-1} \\
                &= \Hmg_{(n,n)}\left[-(\alpha_j(z-x)+n)(z-x)^{n-1}\right]
        \end{split}
    \]
    We now have three cases, depending on the sign of $\alpha_j$. If $\alpha_j > 0$, we have that $\Symb(y\partial_x - \alpha_j)$ is both $(\LHP \cup [a,\infty]) \times (\overline{\UHP} \setminus [a,\infty])$-stable and $(\LHP \cup [-\infty,a]) \times (\overline{\UHP} \setminus [-\infty,a+\frac{n}{\alpha_j}])$-stable for any $a \in \R$. (As usual, we order the variables $(z,w),(x,y)$.) Using Theorem \ref{Jops_graces} and the discussion following the proof, this implies $(y\partial_x - \alpha_j)$ preserves $[a,\infty]$-rootedness and maps $[-\infty,a]$-rooted polynomials to $[-\infty,a+\frac{n}{\alpha_j}]$-rooted polynomials. So, if the (non-infinite) roots of $g$ are contained in the interval $[b,c]$, then the (non-infinite) roots of $(y\partial_x - \alpha_j)g$ are contained in the interval $[b,c+\frac{n}{\alpha_j}]$.
    
    If $\alpha_j < 0$, we have that $\Symb(y\partial_x - \alpha_j)$ is both $(\LHP \cup [a,\infty]) \times (\overline{\UHP} \setminus [a+\frac{n}{\alpha_j},\infty])$-stable and $(\LHP \cup [-\infty,a]) \times (\overline{\UHP} \setminus [-\infty,a])$-stable for any $a \in \R$. As above, this implies $(y\partial_x - \alpha_j)$ preserves $[-\infty,a]$-rootedness and maps $[a,\infty]$-rooted polynomials to $[a+\frac{n}{\alpha_j},\infty]$-rooted polynomials. So, if the (non-infinite) roots of $g$ are contained in the interval $[b,c]$, then the (non-infinite) roots of $(y\partial_x - \alpha_j)g$ are contained in the interval $[b+\frac{n}{\alpha_j},c]$.
    
    Finally for $\alpha_j = 0$, the operator $(y\partial_x - \alpha_j) = y\partial_x$ weakly preserves any interval in which the (non-infinite) roots reside. The main difference for this case is that $y\partial_x$ only preserves weak real-rootedness. Combining these three cases, we are lead to the following root preservation property of $f(\partial_x): \C^n[x] \to \C^n[x]$. Let $\alpha_j^+$ and $\alpha_j^-$ be the positive and negative roots of $f$, respectively. We then have the following, which refers to non-infinite roots:
    \[
        f(\partial_x): [b,c]\text{-rooted} \to \left[b + \sum_j \frac{n}{\alpha_j^-}, c + \sum_j \frac{n}{\alpha_j^+}\right]\text{-rooted}
    \]
    If $f$ has zeros at 0, then $f(\partial_x)$ may map some nonzero $[b,c]$-rooted polynomials to 0. Otherwise, $f(\partial_x)$ is invertible on $\C^n[x]$.
\end{example}

\section*{Acknowledgements}

We would like to thank Nick Ryder for many graduate school conversations on stable polynomials and linear stability preservers. We would also like to thank an anonymous referee for a thorough reading of this paper and many helpful comments.

\bibliographystyle{amsalpha}
\bibliography{apolarity}

\clearpage
\appendix

\section{Tensor Product Decomposition of \texorpdfstring{$\SL_2(\C)$}{SL2(C)} Representations}\label{sl2_app}

In this appendix, we discuss in detail the decomposition of inner tensor products of $\SL_2(\C)$ and $(\SL_2(\C))^m$ representations. The results given here are for the most part standard, and they are typically presented via the theory of Lie groups and algebras (e.g., in \cite{fulton2013reptheory} and \cite{humphreys2012lietheory}). Here though, we discuss these results in terms of the polynomial spaces $V(n)$ and $V(\lambda)$.

That said, the first results we state demonstrate the importance of $V(n)$ and $V(\lambda)$ in the representation theory of $\SL_2(\C)$. In fact, these representations are precisely the irreducible representations of $\SL_2(\C)$ and $(\SL_2(\C))^m$, respectively (see Lecture 11 of \cite{fulton2013reptheory}, and also Proposition 2.3.23 of \cite{kowalski2014introduction}). We will not make full use of this fact but will need the following simpler results.

\begin{proposition}\label{Uirrep_prop}
    For all $n \in \N_0$, we have that $V(n)$ is an irreducible representation of $\SL_2(\C)$ of dimension $n+1$.
\end{proposition}

\begin{proposition}
    For all $\lambda \in \N_0^m$, we have that $V(\lambda) \cong V(\lambda_1) \boxtimes \cdots \boxtimes V(\lambda_m)$ is an irreducible representation of $(\SL_2(\C))^m$ of dimension $\prod_{i}(\lambda_i + 1)$.
\end{proposition}

In particular, outer tensor products of irreducible representations of $\SL_2(\C)$ are irreducible representations of $(\SL_2(\C))^m$. On the other hand, inner tensor products are not irreducible and their decomposition leads to a natural definition of the apolarity form (see \S\ref{inv_maps_subsect}). We now set out to compute these decompositions, which are often given as exercises in the literature (see, e.g., Exercise 11.11 of \cite{fulton2013reptheory}).

\subsection{Decomposition of \texorpdfstring{$V(n) \otimes V(m)$}{V(n) tensor V(m)}}

Fix $n,m \in \N_0$. We now consider the representation of $\SL_2(\C)$ given by the inner tensor product, $V(n) \otimes V(m)$. The importance of the tensor product comes from the fact that it relates to consideration of $\SL_2(\C)$-invariant bilinear forms like the apolarity form. In particular, the decomposition of the tensor product as a sum of irreducible representations (Proposition \ref{Utensor_prop}) will show us exactly how the $D$ map (see Proposition \ref{Dmap_prop}) can be used to define the apolarity form in the representation theoretic context (Definition \ref{Uform_def}).

We begin with an important $\SL_2(\C)$-invariant map.

\begin{proposition}
    Let $x$ and $y$ denote the linear maps defined on $V(k)$ via multiplication by $x$ and $y$, respectively. The linear map $U := (x \otimes y - y \otimes x): V(n) \otimes V(m) \rightarrow V(n+1) \otimes V(m+1)$ is $\SL_2(\C)$-invariant.
\end{proposition}
\begin{proof}
    Straightforward computation, e.g., on simple tensors.
\end{proof}

We then use this $U$ map to show that $D^k$ is not the zero map when $k \leq m,n$.

\begin{lemma}
    For $k \leq m \leq n$, consider the map $DU^k: V(n-k) \otimes V(m-k) \rightarrow V(n-1) \otimes V(m-1)$. We have:
    \[
        DU^k = U^kD + k(n+m-k+1)U^{k-1}
    \]
\end{lemma}
\begin{proof}
    Follows from the fact that $(\partial_x x - x\partial_x)p = p$ and $(\partial_y y - y\partial_y)p = p$.
\end{proof}
% \begin{proof}
%     We first consider the case when $k = 1$. Using $\partial_x x = x \partial_x + I$ and $\partial_y y = y \partial_y + I$ (where $I$ is the identity map), we have:
%     \[
%         \begin{split}
%             DU &= (\partial_x \otimes \partial_y - \partial_y \otimes \partial_x)(x \otimes y - y \otimes x) \\
%                 &= UD + (x\partial_x + y\partial_y) \otimes I + I \otimes (x\partial_x \otimes y\partial_y) + 2 I \otimes I \\
%                 &= UD + (n-1) I \otimes I + (m-1) I \otimes I + 2 I \otimes I \\
%                 &= UD + (n+m)U^0
%         \end{split}
%     \]
%     Inducting on $k$, we obtain:
%     \[
%         DU^k = U^kD + \left(\sum_{j=0}^{k-1} (n-j)+(m-j)\right)U^{k-1}
%     \]
%     Simplification of the sum gives the result.
% \end{proof}

\begin{corollary}\label{nonzero_cor}
    For $k \leq m \leq n$, consider the map $D^kU^k: V(n-k) \otimes V(m-k) \rightarrow V(n-k) \otimes V(m-k)$. We have:
    \[
        D^kU^k(x^{n-k} \otimes x^{m-k}) = \frac{k!(n+m-k+1)!}{(n+m-2k+1)!} (x^{n-k} \otimes x^{m-k}) \neq 0
    \]
    In particular, $x^{n-k} \otimes x^{m-k}$ is in the image of $D^k: V(n) \otimes V(m) \rightarrow V(n-k) \otimes V(m-k)$.
\end{corollary}
\begin{proof}
    Apply the previous lemma $k$ times, and use the fact that $D(x^{n-k} \otimes x^{m-k}) = 0$.
\end{proof}

We will use this fact about the image of $D^k$ to determine the decomposition of $V(n) \otimes V(m)$ into irreducible components. We will also need the following fundamental representation theory result.

\begin{lemma}[Schur's Lemma]
    Let $V,V',W$ be representations of a group $G$, and suppose $V,V'$ are irreducible. Then:
    \begin{enumerate}[label=(\roman*)]
        \item Any $G$-invariant map $\pi: W \rightarrow V$ is either surjective or the zero map.
        \item Any $G$-invariant map $\iota: V \rightarrow W$ is either injective or the zero map.
        \item Any $G$-invariant map $\psi: V \rightarrow V'$ is either an isomorphism or the zero map.
    \end{enumerate}
    Further, if $V,V'$ are vector spaces over an algebraically closed field, then $\psi$ is unique up to scalar.
\end{lemma}

Applying Schur's lemma to each of the $r^\text{th}$ transvectants (discussed at the end of \S\ref{inv_maps_subsect}) yields the desired representation decomposition.

\begin{proposition}\label{Utensor_prop}
    Let $m \leq n$. We have the following decomposition of $V(n) \otimes V(m)$, as a representation of $\SL_2(\C)$, into irreducible components.
    \[
        V(n) \otimes V(m) \cong \bigoplus_{r \leq m} V(n + m - 2r)
    \]
    In particular, $V(n) \otimes V(n) \cong V(2n) \oplus V(2n-2) \oplus \cdots \oplus V(2) \oplus V(0)$.
\end{proposition}
\begin{proof}
    For each $r \in \N_0$, $ \leq m$, consider the $\SL_2(\C)$-invariant $r^\text{th}$ transvectant map:
    \[
        V(n) \otimes V(m) \xrightarrow{D^r} V(n-r) \otimes V(m-r) \xrightarrow{\times} V(n+m-2r)
    \]
    By Corollary \ref{nonzero_cor}, this map is not the zero map, as $x^{n-r} \cdot x^{m-r} = x^{n+m-2r}$ is in its image. Since $V(n+m-2r)$ is irreducible by Proposition \ref{Uirrep_prop}, Schur's lemma implies this map is surjective. This in turn implies $V(n) \otimes V(m) \cong V(n+m-2r) + W_r$ (as a representation) for some subspace $W_r$. Since this holds for all $r \leq m$, we actually have
    \[
        V(n) \otimes V(m) \cong W + \bigoplus_{r \leq m} V(n + m - 2r)
    \]
    for some subspace $W$. The sum of irreducible components here is direct, as any two distinct irreducible components must intersect trivially. To show that we can set $W = 0$, we use the following dimension argument:
    \[
        \sum_{r=0}^m \dim(V(n+m-2r)) = \sum_{r=0}^m (n+m-2r+1) = (n+1)(m+1) = \dim(V(n) \otimes V(m))
    \]
    This completes the proof.
\end{proof}

Along with the stated decomposition, we also obtain something else: the $r^\text{th}$ transvectant is a projection from $V(n) \otimes V(m)$ onto the irreducible component $V(n+m-2r)$. Schur's lemma and the tensor product decomposition then imply this projection is actually unique up to scalar. In a similar way, Schur's lemma also implies $D$ and $U$ must restrict to either a unique isomorphism or the zero map on each irreducible component of $V(n) \otimes V(m)$. In the following, we determine exactly what happens on each component.

\begin{theorem}\label{Umaps_thm}
    Consider the decomposition $V(n) \otimes V(m) \cong V(n+m) \oplus V(n+m-2) \oplus \cdots \oplus V(n-m+2) \oplus V(n-m)$. The maps
    \[
        U: V(n) \otimes V(m) \rightarrow V(n+1) \otimes V(m+1)
    \]
    \[
        D: V(n+1) \otimes V(m+1) \rightarrow V(n) \otimes V(m)
    \]
    restrict to $\SL_2(\C)$-invariant isomorphisms from $V(n+m-2r)$ to $V(n+m-2r)$ for all $0 \leq r \leq m \leq n$. Additionally, $D$ restricts to the zero map on $V(n+m+2)$.
\end{theorem}
\begin{proof}
    By Schur's lemma, the claim immediately follows if $U$ is injective and $D$ is surjective. That $D$ is surjective follows from the fact that the transvectant maps $\times \circ D^r$ are projections onto each of the irreducible components of $V(n) \otimes V(m)$ for $1 \leq r \leq m+1$. That $U$ is injective follows from the fact that $U(v) = 0$ implies $v=0$. One can see this by lexicographically ordering the basis $\{x^jy^{n-j} \otimes x^ky^{m-k} : 0 \leq j \leq n, 0 \leq k \leq m\}$ and considering the highest component of a given $v \in V(n) \otimes V(m)$.
\end{proof}
%     To prove that $U$ is injective, consider the basis $\{x^jy^{n-j} \otimes x^ky^{m-k} : 0 \leq j \leq n, 0 \leq k \leq m\}$ of $V(n) \otimes V(m)$. For any $0 \neq v \equiv \sum_{j,k} a_{j,k} x^jy^{n-j} \otimes x^ky^{m-k} \in V(n) \otimes V(m)$, let $(j_0,k_0)$ be the lexicographically largest pair such that $a_{j_0,k_0} \neq 0$. We then compute:
%     \[
%         \begin{split}
%             Uv &= \sum_{j,k} a_{j,k} x^{j+1}y^{n-j} \otimes x^ky^{m-k+1} - \sum_{j,k} a_{j,k} x^jy^{n-j+1} \otimes x^{k+1}y^{m-k} \\
%                 &= \sum_{j,k} (a_{j,k} - a_{j+1,k-1}) x^{j+1}y^{n-j} \otimes x^ky^{m-k+1} \\
%                 &= (a_{j_0,k_0}-0) x^{j_0+1}y^{n-j_0} \otimes x^{k_0}y^{m-k_0+1} + \text{other terms}
%         \end{split}
%     \]
%     In particular $Uv \neq 0$, and so $U$ is injective.
% \end{proof}

Our main application of this theory is given as follows. Consider the $n^\text{th}$ transvectant map $\times \circ D^n: V(n) \otimes V(n) \rightarrow V(0) \cong \C$, which is nonzero by the previous theorem. This map can be interpreted as an $\SL_2(\C)$-invariant bilinear form on $V(n)$. It turns out that the apolarity bilinear form used in Grace's theorem also has this property, and this justifies the following definition.

\newtheorem*{Uform_def}{Definition \ref{Uform_def}}
\begin{Uform_def}
    We call the $n^\text{th}$ transvectant
    \[
        V(n) \otimes V(n) \xrightarrow{D^n} V(0) \otimes V(0) \xrightarrow{\times} V(0) \cong \C
    \]
    the \emph{apolarity form} of $V(n)$.
\end{Uform_def}

\begin{corollary}\label{Uunique_cor}
    The apolarity form is the unique (up to scalar) nondegenerate $\SL_2(\C)$-invariant bilinear form on $V(n)$.
\end{corollary}

\subsection{Decomposition of \texorpdfstring{$V(\lambda) \otimes V(\mu)$}{V(lambda) tensor V(mu)}}

Fix $\lambda,\mu \in \N_0^m$, and let $V(\lambda)$ and $V(\mu)$ denote the irreducible representations of $(\SL_2(\C))^m$ given by the outer tensor products:
\[
    V(\lambda) \cong V(\lambda_1) \boxtimes \cdots \boxtimes V(\lambda_m)
    \qquad\qquad\qquad
    V(\mu) \cong V(\mu_1) \boxtimes \cdots \boxtimes V(\mu_m)
\]
We next generalize the above results to the inner tensor product of these two representations, $V(\lambda) \otimes V(\mu)$. In particular, we determine the decomposition of this tensor product and define a multivariate apolarity form. Note that these statements strictly generalize the previous analogous statements.

\begin{proposition}[c.f. Proposition \ref{Utensor_prop}]\label{Mtensor_prop}
    Let $\mu \leq \lambda$. We have the following decomposition of $V(\lambda) \otimes V(\mu)$, as a representation of $(\SL_2(\C))^m$, into irreducible components.
    \[
        V(\lambda) \otimes V(\mu) \cong \bigoplus_{\alpha \leq \mu} V(\lambda + \mu - 2\alpha)
    \]
\end{proposition}
\begin{proof}
    We compute:
    \[
        \begin{split}
            V(\lambda) \otimes V(\mu) &\cong \big(V(\lambda_1) \boxtimes \cdots \boxtimes V(\lambda_m)\big) \otimes \big(V(\mu_1) \boxtimes \cdots \boxtimes V(\mu_m)\big) \\
                &\cong \big(V(\lambda_1) \otimes V(\mu_1)\big) \boxtimes \cdots \boxtimes \big(V(\lambda_m) \otimes V(\mu_m)\big) \\
                &\cong \left(\bigoplus_{\alpha_1 \leq \mu_1} V(\lambda_1 + \mu_1 - 2\alpha_1)\right) \boxtimes \cdots \boxtimes \left(\bigoplus_{\alpha_m \leq \mu_m} V(\lambda_m + \mu_m - 2\alpha_m)\right) \\
                &\cong \bigoplus_{\alpha \leq \mu} V(\lambda + \mu - 2\alpha)
        \end{split}
    \]
    The last step uses the distributive law for sums and tensor products of representations.
\end{proof}

\begin{theorem}[c.f. Theorem \ref{Umaps_thm}]
    For any $\beta \in \N_0^m$, define $U^\beta := U^{\beta_1} \boxtimes \cdots \boxtimes U^{\beta_m}$. Define $D^\beta$ similarly. For any $\mu \leq \lambda \in \N_0^m$, the maps
    \[
        U^\beta: V(\lambda) \otimes V(\mu) \rightarrow V(\lambda + \beta) \otimes V(\mu + \beta)
    \]
    \[
        D^\beta: V(\lambda + \beta) \otimes V(\mu + \beta) \rightarrow V(\lambda) \otimes V(\mu)
    \]
    restrict to $\SL_2(\C)$-invariant isomorphisms on the components of $V(\lambda) \otimes V(\mu) \cong \bigoplus_{\alpha \leq \mu} V(\lambda + \mu - 2\alpha)$. Finally, $D^\beta$ restricts to the zero map on the other irreducible components of $V(\lambda + \beta) \otimes V(\mu + \beta)$.
\end{theorem}
\begin{proof}
    Follows by induction on $\beta$, using Theorem \ref{Umaps_thm}.
\end{proof}

\newtheorem*{Mform_def}{Definition \ref{Mform_def}}
\begin{Mform_def}[c.f. Definition \ref{Uform_def}]
    We call the map
    \[
        V(\lambda) \otimes V(\lambda) \xrightarrow{D^\lambda} V(0^m) \otimes V(0^m) \xrightarrow{\times} V(0^m) \cong \C
    \]
    the \emph{apolarity form} of $V(\lambda)$.
\end{Mform_def}

\begin{corollary}[c.f. Corollary \ref{Uunique_cor}]\label{Munique_cor}
    The apolarity form is the unique (up to scalar) nondegenerate $(\SL_2(\C))^m$-invariant bilinear form on $V(\lambda)$.
\end{corollary}

\section{The Grace-Walsh-Szeg\texorpdfstring{\H{o}}{o} Coincidence Theorem}\label{gws_app}

A classical result in the representation theory of $\SL_2(\C)$ is the fact that $V(n) \cong \operatorname{Sym}^n(V(1))$. Here, $\operatorname{Sym}^n(V(1))$ denotes the set of symmetric tensors in $V(1)^{\otimes n}$, or alternatively, the set of symmetric elements in $V(1^n)$. That is, there is some $\SL_2(\C)$-invariant injection from $V(n)$ to $V(1)^{\otimes n}$, and by our conceptual thesis this map should transfer stability information. In fact, this idea is formalized in the Grace-Walsh-Szeg{\H{o}} coincidence theorem, and the injective map is known as the polarization map.

\subsection{Polarization and Projection}

For polynomials of degree $m \leq n$, the degree-$n$ polarization map is defined on monomials as follows and is extended linearly.
\[
    \begin{split}
        \Pi_n^\uparrow : \C^n[x] &\rightarrow \C^{(1^n)}[x_1,\ldots,x_n] \\
        x^k &\mapsto \frac{1}{n!} \sum_{\sigma \in S_n} \prod_{j=1}^k x_{\sigma(j)}
    \end{split}
\]
This definition can be extended to homogeneous polynomials in $V(n)$ by composing with $\Hmg_n^{-1}$ and $\Hmg_{(1^n)}$. The map $\Pi_n^\uparrow$ has a left inverse $\Pi_n^\downarrow$, called the projection map, which we define as follows.
\[
    \begin{split}
        \Pi_n^\downarrow : \C^{(1^n)}[x_1,\ldots,x_n] &\rightarrow \C^n[x] \\
        f(x_1,x_2,\ldots,x_n) &\mapsto f(x,x,\ldots,x)
    \end{split}
\]
That is, $\Pi_n^\downarrow \circ \Pi_n^\uparrow$ is the identity map. Similarly, this definition can be extended to homogeneous polynomials by composing with $\Hmg_{(1^n)}^{-1}$ and $\Hmg_n$.

It is well-known that $\Pi_n^\uparrow$ is an injective linear map onto the subspace of symmetric multi-affine polynomials. This fact then extends to homogeneous polynomials, where the terms \emph{symmetric} and \emph{multi-affine} each refer to pairs of homogeneous variables. Further, one can define multivariate polarization and projection maps via composition: $\Pi_\lambda^\uparrow := \Pi_{\lambda_m}^\uparrow \circ \cdots \circ \Pi_{\lambda_1}^\uparrow$ and $\Pi_\lambda^\downarrow := \Pi_{\lambda_m}^\downarrow \circ \cdots \circ \Pi_{\lambda_1}^\downarrow$. Injectivity then automatically extends to $\Pi_\lambda^\uparrow$, and $\Pi_\lambda^\downarrow \circ \Pi_\lambda^\uparrow$ is the identity map.

These two maps arise naturally in the theory of polynomials in general, and play an important role in the theory of stability, via the Grace-Walsh-Szeg{\H{o}} coincidence theorem as well as in the proof of the Borcea-Br{\"a}nd{\'e}n characterization of linear operators. The next result shows they also have represention theoretic importance.

\begin{proposition}\label{pol_proj_inv_prop}
    Fix $\lambda \in \N_0^m$, and view $V(\lambda) \cong V(\lambda_1) \boxtimes \cdots \boxtimes V(\lambda_m)$ and $V(1^\lambda) \cong V(1)^{\otimes \lambda_1} \boxtimes \cdots \boxtimes V(1)^{\otimes \lambda_m}$ as representations of $(\SL_2(\C))^m$. The maps $\Pi_\lambda^\uparrow: V(\lambda) \rightarrow V(1^\lambda)$ and $\Pi_\lambda^\downarrow: V(1^\lambda) \rightarrow V(\lambda)$ are $(\SL_2(\C))^m$-invariant.
\end{proposition}
\begin{proof}
    Note that by proving the result for $\Pi_n^\uparrow$ and $\Pi_n^\downarrow$ with $m=1$, the general result follows since $\Pi_\lambda^\uparrow$ and $\Pi_\lambda^\downarrow$ are compositions of such maps. To prove it for $m=1$, note that the set of symmetric elements in $V(1)^{\otimes n} \cong V(1^n)$ is invariant under the diagonal action of $\SL_2(\C)$. Further, since $V(n)$ is irreducible of dimension $n+1$ and $V(1)^{\otimes n}$ has a single irreducible component of dimension $n+1$, Schur's lemma implies the result.
\end{proof}

This result then has a few corollaries which will help to shed light on results related to polarization and the apolarity form. The first will be useful in elucidating the representation theoretic ties to the Grace-Walsh-Szeg{\H{o}} coincidence theorem below.

\begin{lemma}\label{form_pol_commute_lemma}
    Fix $\lambda \in \N_0^m$. Then the apolarity form commutes with polarization up to scalar. That is:
    \[
        D^\lambda = D^{(1^\lambda)} \circ (\Pi_\lambda^\uparrow \otimes \Pi_\lambda^\uparrow)
    \]
\end{lemma}
\begin{proof}
    The map $D^{(1^\lambda)} \circ (\Pi_\lambda^\uparrow \otimes \Pi_\lambda^\uparrow): V(\lambda) \otimes V(\lambda) \rightarrow \C$ is an $(\SL_2(\C))^m$-invariant bilinear form on $V(\lambda)$. By uniqueness (see Corollary~\ref{Munique_cor}), this then must equal $D^\lambda$ up to scalar.
\end{proof}

The content of this result is that fact that we have commutativity even though $D^{(1^\lambda)}$ is a priori the apolarity form with respect to a different group action than that of $D^\lambda$ (i.e., $(\SL_2(\C))^{|\lambda|}$ instead of $(\SL_2(\C))^m$). That said, it should be noted that the analogous commutativity statement with the projection map $\Pi_\lambda^\downarrow$ does \emph{not} hold (unless of course, one restricts to the image of $\Pi_\lambda^\uparrow$).

The purpose of this result is then to demonstrate the connection between a polynomial and its polarization. In particular, if Grace's theorem gives stability information via the apolarity form, then the previous result shows that the polarizations of those polynomials will have the same stability information. We prove this rigorously in Corollary \ref{gws_cor}.

Proposition \ref{pol_proj_inv_prop} also leads to one of the crucial results used in the proof of the Borcea-Br{\"a}nd{\'e}n characterization of linear operators (Lemma 2.5 in \cite{bb1}). It relies on the notion of ``the polarization of an operator'', given by $T \mapsto \Pi_\alpha^\uparrow \circ T \circ \Pi_\lambda^\downarrow$ (see \S2.2 in \cite{bb1}). We do not make explicit use of this result, but we state it here to demonstrate that operator polarization has a representation theoretic interpretation similar to that of the usual polynomial polarization.

\begin{proposition}
    The symbol of the polarization of an operator $T$ is the polarization of the symbol of $T$.
\end{proposition}
\begin{proof}
    Using Proposition \ref{pol_proj_inv_prop} and Definition \ref{symb_def}, it is straightforward to see that all the maps involved are injective $(\SL_2(\C))^{2m}$-invariant linear maps (i.e., polarization of polynomials, polarization of operators, the $\Symb$ map). The result then follows from a dimension argument and Schur's lemma, in a way similar to that of the proof of Proposition \ref{pol_proj_inv_prop}.
\end{proof}
    
\subsection{The Coincidence Theorem}

The Grace-Walsh-Szeg{\H{o}} coincidence theorem has strong ties to Grace's theorem, and most books and surveys on the subject state the two results side by side. Some books (e.g., \cite{rahman2002anthpoly}) even go so far as to demonstrate their equivalence, perhaps with other results involving typical polynomial convolutions. Here, we will state and prove the general multivariate version of the theorem in terms of homogeneous polynomials, making use of evaluation symbols and Grace's theorem (Theorem \ref{graces_thm}).

First though, consider the following corollary to the symbol lemma (Lemma \ref{symb_lemma}) which is similar in spirit to the evaluation symbol lemma (Lemma \ref{ev_lemma}). Note that when applied to $p \in V(n)$ with $m=1$, this result has the following intuitive statement as a corollary: \emph{$D^n(q \otimes p)$ is equal to the evaluation of $\Pi_n^\uparrow p$ at the roots of $q$}.

\begin{lemma}
    Fix $\lambda \in \N_0^m$, $p \in V(\lambda)$, and any $(a:b) \in (\CP^1)^{|\lambda|}$, more explicitly defined as follows:
    \[
        (a:b) \equiv \big((a_{1,1}:b_{1,1}), \ldots, (a_{1,\lambda_1}:b_{1,\lambda_1}), \ldots, (a_{m,1}:b_{m,1}), \ldots, (a_{m,\lambda_m}:b_{m,\lambda_m})\big) \in (\CP^1)^{\lambda_1 + \cdots + \lambda_m}
    \]
    We have the following:
    \[
        (\Pi_\lambda^\uparrow p)(a,b) = D^\lambda\left(\Pi_\lambda^\downarrow\big(\Symb(\ev_{(a,b)})\big) \otimes p\right) = D^\lambda\left(\prod_{k=1}^m \prod_{j=1}^{\lambda_k} (b_{k,j}x_k - a_{k,j}y_k) \otimes p \right)
    \]
\end{lemma}
\begin{proof}
    We first prove the case of $m=1$ and $p \in V(n)$, given as follows:
    \[
        (\Pi^\uparrow_n p)(a,b) = D^n\left(\Pi_n^\downarrow\big(\Symb(\ev_{(a,b)})\big) \otimes p\right) = D^n\left(\prod_{j=1}^n (b_jx - a_jy) \otimes p \right)
    \]
    The second equality follows immediately from the definition of $\Pi_n^\downarrow$ and of $\ev_{(a,b)}$ (Definition \ref{ev_def}). For the first equality, Lemma \ref{form_pol_commute_lemma} implies:
    \[
        D^n\left(\Pi_n^\downarrow\big(\Symb(\ev_{(a,b)})\big) \otimes p\right) = D^{(1^n)}\left(\Pi_n^\uparrow \circ \Pi_n^\downarrow\big(\Symb(\ev_{(a,b)})\big) \otimes \Pi_n^\uparrow p\right)
    \]
    By definition of $\Pi_n^\uparrow$, both $\Pi_n^\uparrow \circ \Pi_n^\downarrow\big(\Symb(\ev_{(a,b)})\big)$ and $\Pi_n^\uparrow p$ are symmetric in pairs of variables. Further, we can explicitly compute:
    \[
        \Pi_n^\uparrow \circ \Pi_n^\downarrow\big(\Symb(\ev_{(a,b)})\big) = \frac{1}{n!} \sum_{\sigma \in S_n} \prod_{j=1}^n (b_jx_{\sigma(j)} - a_jy_{\sigma(j)}) %= \frac{1}{n!} \sum_{\sigma \in S_n} \prod_{j=1}^n (b_{\sigma(j)}x_j - a_{\sigma(j)}y_j)
    \]
    This expression and the fact that $\Pi_n^\uparrow p$ is symmetric then imply:
    \[
    \begin{split}
        D^{(1^n)}\left(\Pi_n^\uparrow \circ \Pi_n^\downarrow\big(\Symb(\ev_{(a,b)})\big) \otimes \Pi_n^\uparrow p\right) &= \frac{1}{n!} \sum_{\sigma \in S_n} D^{(1^n)}\left(\prod_{j=1}^n (b_jx_{\sigma(j)} - a_jy_{\sigma(j)}) \otimes \Pi_n^\uparrow p\right) \\
            &= D^{(1^n)}\left(\prod_{j=1}^n (b_jx_j - a_jy_j) \otimes \Pi_n^\uparrow p\right)
    \end{split}
    \]
    The last expression is then equal to $(\Pi^\uparrow_n p)(a,b)$ by the evaluation symbol lemma (Lemma \ref{ev_lemma}).

    For the general case of $p \in V(\lambda)$, the same argument applies with a few tweaks. The key change is that $\Pi_\lambda^\uparrow \circ \Pi_\lambda^\downarrow\big(\Symb(\ev_{(a,b)})\big)$ and $\Pi_\lambda^\uparrow p$ are no longer fully symmetric in pairs of variables, but instead invariant under the action of $S_{\lambda_1} \times \cdots \times S_{\lambda_m}$ on the $|\lambda|$ pairs of variables.
    %
    % For the first equality, Lemma \ref{form_pol_commute_lemma} implies:
    % \[
    %     D^\lambda\left(\Pi_\lambda^\downarrow\big(\Symb(\ev_{(a,b)})\big) \otimes p\right) = D^{(1^\lambda)}\left(\Pi_\lambda^\uparrow \circ \Pi_\lambda^\downarrow\big(\Symb(\ev_{(a,b)})\big) \otimes \Pi_\lambda^\uparrow p\right)
    % \]
    % Since $\Pi_\lambda^\uparrow \circ \Pi_\lambda^\downarrow\big(\Symb(\ev_{(a,b)})\big)$ is the symmetrization of $\Symb(\ev_{(a,b)})$ in each set of $\lambda_k$ pairs of variables, we can write it as an average over $S_{\lambda_1} \times \cdots \times S_{\lambda_m}$ (product of symmetric groups) of $\Symb(\ev_{(a,b)})$ with permuted variables. Since $\Pi_\lambda^\uparrow p$ is symmetric in each set of $\lambda_k$ pairs of variables, we then have:
    % \[
    %     D^{(1^\lambda)}\left(\Pi_\lambda^\uparrow \circ \Pi_\lambda^\downarrow\big(\Symb(\ev_{(a,b)})\big) \otimes \Pi_\lambda^\uparrow p\right) = D^{(1^\lambda)}\left(\Symb(\ev_{(a,b)}) \otimes \Pi_\lambda^\uparrow p\right)
    % \]
    % The result then follows from the evaluation symbol lemma (Lemma \ref{ev_lemma}).
\end{proof}

Generally speaking, the above lemma demonstrates the strong connection between the apolarity form and the polarization map. We now utilize this to prove the coincidence theorem.

\begin{corollary}[Grace-Walsh-Szeg{\H{o}}]\label{gws_cor}
    Fix $\lambda \in \N_0^m$, $p \in V(\lambda)$, and any disjoint Grace pair $(C_1 \times \cdots \times C_m, B_1 \times \cdots \times B_m)$. If $p$ is $(C_1 \times \cdots \times C_m)$-stable, then $\Pi_\lambda^\uparrow p$ is $(C_1^{\lambda_1} \times \cdots \times C_m^{\lambda_m})$-stable.
\end{corollary}
\begin{proof}
    So as to prove the contrapositive, suppose $\Pi_\lambda^\uparrow p$ is not $(C_1^{\lambda_1} \times \cdots \times C_m^{\lambda_m})$-stable. That is, suppose $(\Pi_\lambda^\uparrow p)(a,b) = 0$ for some $(a,b) = (a_1,b_1,\ldots,a_{m,\lambda_m},b_{m,\lambda_m}) \in \C^{2|\lambda|}$ such that $(a_{j,k}:b_{j,k}) \in C_j$ for all $j \in [m]$ and $k \in [\lambda_j]$. By the previous lemma, this implies:
    \[
        D^\lambda\left(\Pi_\lambda^\downarrow\big(\Symb(\ev_{(a,b)})\big) \otimes p\right) = 0
    \]
    By disjointness of $C_j$ and $B_j$ for all $j \in [m]$, we then have that $(a_{j,k}:b_{j,k}) \not\in B_j$ for all $j,k$. Therefore $\Symb(\ev_{(a,b)})$ is $(B_1^{\lambda_1} \times \cdots \times B_m^{\lambda_m})$-stable (see Definition \ref{ev_def}). This implies $\Pi_\lambda^\downarrow\big(\Symb(\ev_{(a,b)})\big)$ is $(B_1 \times \cdots \times B_m)$-stable. Since $D^\lambda\left(\Pi_\lambda^\downarrow\big(\Symb(\ev_{(a,b)})\big) \otimes p\right) = 0$, the definition of Grace pair (Definition \ref{grace_pair_def}) then implies $p$ must not be $(C_1 \times \cdots \times C_m)$-stable.
\end{proof}

By Theorem \ref{grace_pairs_thm}, this implies the coincidence theorem for circular regions when $m > 1$ and for any projectively convex regions when $m = 1$. Notice that there is no reference made to degree or convexity restrictions (compare this to Theorems 1.1 and 1.2 in \cite{bb2}). As discussed above, this is one of the main benefits of using homogeneous polynomials and interpreting zeros as lying in $\CP^1$ and $(\CP^1)^m$.

\end{document}